\documentclass[12pt,reqno]{amsart}

\usepackage{amsmath,amssymb,ifthen}
\usepackage{fullpage}
\usepackage{graphicx,psfrag,subfigure}
\usepackage{color}

\usepackage{moreverb}

\def\H{\widetilde{H}}

\def\R{{\mathbb R}}

\def\EE{{\mathcal E}}

\def\KK{{\mathcal K}}

\def\MM{{\mathcal M}}

\def\PP{{\mathcal P}}

\def\SS{{\mathcal S}}
\def\SSS{{\mathbb S}}
\def\TT{{\mathcal T}}
\def\NN{\mathcal{N}}

\def\XX{{\mathcal X}}
\def\YY{{\mathcal Y}}
\def\II{{\mathcal I}}
\def\L{\mathbf{L}}
\def\H{\mathbf{H}}

\newcommand{\norm}[3][]{#1\|#2#1\|_{#3}}

\def\set#1#2{\big\{#1\,:\,#2\big\}}

\def\eps{\varepsilon}

\newcounter{const}
\def\setc#1{\refstepcounter{const}\label{const:#1}C_{\theconst}}
\def\c#1{C_{\ref{const:#1}}}

\newcounter{contractionnumber}
\def\nameq#1#2{%
  \ifthenelse{\equal{#1}{reduction}}{q_{\rm red}}{%
  \ifthenelse{\equal{#1}{estconv}}{q_{\rm est}}{%
  \ifthenelse{\equal{#1}{cea}}{q_{\mbox{\scriptsize C\'ea}}}{%
  \ifthenelse{\equal{#2}{newcounter}}{\refstepcounter{contractionnumber}\label{contraction#1}}{}q_{\ref{contraction#1}}}%
}}}

\def\FFF{\mathbf{F}}
\def\GGG{\mathbf{G}}
\def\www{\mathbf{w}}

\def\mmm{\mathbf{m}}
\def\EEE{\mathbf{E}}
\def\HHH{\mathbf{H}}
\def\JJJ{\mathbf{J}}
\def\heff{\mathbf{H}_{\text{eff}}}

\def\subsub{\Subset}
\def\nnn{\mathbf{n}}
\def\zzz{\mathbf{z}}
\def\xxx{\mathbf{x}}
\def\vvv{\mathbf{v}}
\def\uuu{\mathbf{u}}
\def\pphi{\pmb{\phi}}
\def\PPHI{\pmb{\Phi}}
\def\ppsi{\pmb{\psi}}
\def\PPSI{\pmb{\Psi}}
\def\vphi{\pmb{\varphi}}
\def\zzeta{\pmb{\zeta}}

\def\TTHETA{\pmb{\Theta}}
\def\curl{\text{\textbf{curl}}}
\def\wgamma{\gamma_{hk}}
\DeclareMathOperator{\diver}{div}

\newtheorem{theorem}{Theorem}

\newtheorem{lemma}[theorem]{Lemma}

\newtheorem{algorithm}[theorem]{Algorithm}
\newtheorem{definition}[theorem]{Definition}
\newenvironment{remark}{\medskip\noindent\textbf{Remark.}\ \it}{\qed\smallskip}

\def\subsection#1
{
 \bigskip

 \refstepcounter{subsection}
 {\noindent\bf\arabic{section}.\arabic{subsection}.~#1.~}
}

\begin{document}

\title[Convergent scheme for Maxwell-LLG]%
{A convergent linear finite element scheme for the Maxwell-Landau-Lifshitz-Gilbert equation}
\date{\today}

\author{L'.~Ba\v nas}
\author{M.~Page}
\author{D.~Praetorius}
\address{Institute for Analysis and Scientific Computing,
       Vienna University of Technology,
       Wiedner Hauptstra\ss{}e 8-10,
       A-1040 Wien, Austria}
\email{L.Banas@hw.ac.uk}
\email{Dirk.Praetorius@tuwien.ac.at}
\email{Marcus.Page@tuwien.ac.at\quad\rm(corresponding author)}


\begin{abstract}
We consider a lowest-order finite element discretization of
the nonlinear system of Maxwell's and Landau-Lifshitz-Gilbert equations (MLLG). Two algorithms are proposed to numerically solve this problem, both of which only require the solution of at most two linear systems per timestep. One of the algorithms is fully decoupled in the sense that each timestep consists of the sequential computation of the magnetization and \emph{afterwards} the magnetic and electric field. Under some mild assumptions on the effective field, we show that both algorithms converge towards weak solutions of the MLLG system.
Numerical experiments for a micromagnetic benchmark problem demonstrate the performance
of the proposed algorithms.
\end{abstract}


\maketitle

\section{Introduction}\label{sec:0}

\noindent
The understanding of magnetization dynamics, especially on a microscale, is of utter relevance, for example in the development of magnetic sensors, recording heads, and magneto-resistive storage devices. In the literature, a well accepted model for micromagnetic phenomena, is the Landau-Lifshitz-Gilbert equation (LLG), see~\eqref{eq:mllg1}. This nonlinear partial differential equation describes the behaviour of the magnetization of some ferromagnetic body under the influence of a so-called effective field. Existence (and non-uniqueness) of weak solutions of LLG goes back to~\cite{as, visintin}. Existence of weak solutions for MLLG was first shown in~\cite{carbou}. For a complete review of the analysis for LLG, we refer to~\cite{cimrak, gc, mp06} or the monographs~\cite{hubertschaefer, prohl} and the references therein. As far as numerical simulation is concerned, convergent integrators can be found e.g.\ in the works~\cite{bp, bjp} or~\cite{banas}, where the latter considers a weak integrator for 
the coupled MLLG system. From the viewpoint of numerical analysis, the integrator from~\cite{bjp} suffers from explicit time stepping, since this imposes a strong coupling of the timestep-size $k$ and the spatial mesh-size $h$. The integrators of~\cite{bp, banas}, on the other hand, rely on the implicit midpoint rule for time discretization, and unconditional convergence is proved. In practice though, a nonlinear system of equations has to be solved in each timestep, and to that end, a fixed-point iteration is proposed in the works~\cite{bp,banas}. This, however, again leads to a coupling of $h$ and $k$, and thus destroys unconditional convergence. Using 
the Newton method for the midpoint scheme seems to allow for larger time-steps, see \cite{DSM1} and also \cite{sllg}, where an efficient Newton-multigrid nonlinear solver has been proposed. 

In~\cite{alouges2008}, an unconditionally convergent projection-type integrator is proposed, which, despite the nonlinearity of LLG, only requires the solution of one linear system per timestep. The effective field in this work, however, only covers microcrystalline exchange effects and is thus quite restricted. In the subsequent works~\cite{alouges2011, mathmod2012, gamm2011} the analysis for this integrator was widened to cover more general (linear) field contributions, where only the highest-order exchange contribution is treated implicitly, whereas the other contributions are treated explicitly. This allows to minimize computational effort while still maintaing unconditional convergence. Finally, in the very recent work~\cite{multiscale}, the authors could show unconditional convergence of this integrator, where the effective field consists of some general 
energy contributions, which are only supposed to fulfill a certain set of properties. 
This particularly covers some nonlinear contributions, as well as certain multiscale problems. In addition, it is shown in~\cite{multiscale} that errors arising due to approximate computation of field contributions like e.g.\ the demagnetizing field can be incorporated into the analysis. 

In~\cite{alouges2012}, the authors also investigate a higher-order extension of this algorithm which, however, requires implicit treatment of nonlocal contributions like the magnetostatic strayfield.

In our work, we extend the analysis of the aforementioned works and show that the integrator from~\cite{alouges2008} can be coupled with a weak formulation of the full Maxwell system~\eqref{eq:mllg2}--\eqref{eq:mllg3}. For the integration of this system, we propose two algorithms that only require the solution of one (Algorithm~\ref{alg:1}) resp.\ two linear systems (Algorithm~\ref{alg:2}) per timestep while still guaranteeing unconditional convergence (Theorem~\ref{thm:convergence}). 
The contribution of the present work can be summarized as follows:
\begin{itemize}
\item We extend the linear integrator from~\cite{alouges2008, mathmod2012} to time-dependent contributions of the effective field by considering the full Maxwell equations instead of the magnetostatic simplification.
\item Unlike~\cite{banas}, at most two linear systems per timestep, instead of a coupled nonlinear system, need to be solved. Nevertheless, we still prove unconditional convergence.
\item Unlike~\cite{banas}, the decoupling of the Maxwell and the LLG part in the integrator of Algorithm~\ref{alg:2} is rigorously included into the convergence analysis of the time-marching scheme.
\end{itemize}

\textbf{Outline.}
The remainder of this paper is organized as follows: In Section~\ref{sec:intro}, we recall the mathematical model for the full Maxwell-LLG system (MLLG) and recall the notion of a weak solution (Definition~\ref{def:weak_sol}). In Section~\ref{sec:prelim}, we collect some notation and preliminaries, as well as the definition of the discrete ansatz spaces and their corresponding interpolation operators. In Section~\ref{sec:algo}, we propose two algorithms (Algorithm~\ref{alg:1} and~\ref{alg:2}) to approximate the MLLG system numerically. The large 
Section~\ref{sec:convergence} is then devoted to our main convergence result (Theorem~\ref{thm:convergence}) and its proof. Finally, in Section~\ref{sec:numerics}, some numerical results conclude this work.

\section{Model Problem}\label{sec:intro}

We consider the Maxwell-Landau-Lifshitz-Gilbert equation (MLLG) which describes the evolution of the magnetization of a ferromagnetic body that occupies the domain $\omega \subsub \Omega \subseteq \R^3$. For a given damping parameter $\alpha > 0$, the magnetization $\mmm :(0,T) \times \omega \rightarrow \SSS^2$ and the electric and magnetic fields $\EEE, \HHH:(0,T) \times \Omega \rightarrow \R^3$ satisfy the MLLG system
\begin{subequations}\label{eq:mllg}
\begin{align}
&\mmm_t - \alpha \mmm \times \mmm_t = - \mmm \times \heff \quad \text{ in } \omega_T := (0,T) \times \omega\label{eq:mllg1}\\
&\eps_0 \EEE_t - \nabla \times \HHH + \sigma \chi_\omega\EEE = -\JJJ \quad \text{ in } \Omega_T := (0,T) \times \Omega\label{eq:mllg2}\\
&\mu_0 \HHH_t + \nabla \times \EEE = -\mu_0 \mmm_t \quad \text{ in } \Omega_T,\label{eq:mllg3}
\end{align}
where the effective field $\heff$ consists of $\heff = C_e \Delta \mmm + \HHH + \pi(\mmm)$ for some general energy contribution $\pi$ which is assumed to fulfill a certain set of properties, see~\eqref{eq:assum2}--\eqref{eq:assum3}. This is in analogy to~\cite{multiscale}. We stress that, with the techniques from~\cite{multiscale}, an approximation $\pi_h$ of $\pi$ can be included into the analysis, as well. 
We emphasize that throughout this work, the case $\heff = C_e \Delta \mmm + \HHH + C_a D\Phi(\mmm) + \HHH_{ext}$ is particularly covered. Here, $\Phi(\cdot)$ denotes the crystalline anisotropy density and $\HHH_{ext}$ is a given applied field. The constants $\eps_0, \mu_0 \ge 0$ denote the electric and magnetic permeability of free space, respectively, and the constant $\sigma \ge 0$ stands for the conductivity of the ferromagnetic domain $\omega$. The field $\JJJ : \Omega_T \rightarrow \R^3$ describes an applied current density and $\chi_\omega : \Omega \rightarrow \{0,1\}$ is the characteristic function of $\omega$. As is usually done for simplicity, we assume $\Omega \subset \R^3$ to be bounded with perfectly conducting outer surface $\partial \Omega$ into which the ferromagnet $\omega \subsub \Omega$ is embedded, and $\Omega \backslash \overline \omega$ is assumed to be vacuum. In addition, the MLLG system~\eqref{eq:mllg} is supplemented by initial conditions
\begin{align}\label{eq:init}
\mmm(0,\cdot) = \mmm^0  \text{ in } \omega \quad \text{and} \quad \EEE(0,\cdot) = \EEE^0, \quad \HHH(0,\cdot) = \HHH^0  \text{ in } \Omega
\end{align}
as well as boundary conditions
\begin{align}\label{eq:boundary}
\partial_\nnn\mmm = 0  \text{ on } \partial \omega_T, \qquad \EEE\times\nnn = 0  \text{ on } \partial \Omega_T.
\end{align}
%
Note that the side constraint $|\mmm| = 1$ a.e.\ in $\omega_T$ does not need to be enforced explicitely, but follows from $|\mmm^0| = 1$ a.e.\ in $\omega$ and $\partial_t |\mmm|^2 = 2 \mmm \cdot \mmm_t = 0$ in $\omega_T$, which is a consequence of~\eqref{eq:mllg1}. This behaviour should also be reflected by the numerical integrator.
In analogy to \cite{carbou, banas}, we assume the given data to satisfy
\begin{align}\label{eq:data}
\mmm^0 \in H^1(\omega, \SSS^2), \qquad \HHH^0, \EEE^0 \in \L^2(\Omega, \R^3), \qquad \JJJ \in \L^2(\Omega_T, \R^3)
\end{align}
as well as
%
\begin{align}\label{eq:consistency}
\diver(\HHH^0 + \chi_\omega \mmm^0) = 0 \quad \text{ in } \Omega, \qquad \langle\HHH^0 + \chi_\omega\mmm^0, \nnn\rangle = 0 \quad \text{ on } \partial \Omega.
\end{align}
\end{subequations}
With the space 
\begin{align*}
\HHH_0(\curl, \Omega) := \set{\varphi \in \L^2(\Omega)}{\nabla \times \varphi \in \L^2(\Omega), \varphi \times \nnn = 0 \text{ on } \Gamma},
\end{align*}
we now recall the notion of a weak solution of~\eqref{eq:mllg1}--\eqref{eq:mllg3} from~\cite{carbou}.
\begin{definition}\label{def:weak_sol}
Given~\eqref{eq:data}--\eqref{eq:consistency}, the tupel $(\mmm, \EEE, \HHH)$ is called a weak solution of MLLG if,
\begin{itemize}
\item[(i)] $\mmm \in \H^1(\omega_T)$ with $|\mmm| = 1$ almost everywhere in $\omega_T$ and $(\EEE, \HHH) \in \L^2(\Omega_T)$;
\item[(ii)] for all $\vphi \in C^\infty(\omega_T)$ and $\zzeta \in C_c^\infty\big([0,T); C^\infty(\Omega) \cap \HHH^0(\emph{\curl}, \Omega)\big),$ we have
\begin{align}\label{eq:weak_sol1}
\int_{\omega_T} \langle\mmm_t, \vphi\rangle - \alpha \int_{\omega_T} \langle(\mmm \times \mmm_t), \vphi\rangle &= -C_e\int_{\omega_T} \langle(\nabla \mmm\times\mmm), \nabla\vphi\rangle\\\nonumber
&\quad+ \int_{\omega_T}\langle(\HHH \times \mmm),\vphi\rangle + \int_{\omega_T}\langle(\pi(\mmm)\times \mmm), \vphi\rangle, \\
-\eps_0\int_{\Omega_T}\langle\EEE, \zzeta_t\rangle - \int_{\Omega_T}\langle\HHH,\nabla \times \zzeta\rangle
+\sigma \int_{\omega_T}\langle\EEE, \zzeta\rangle &= -\int_{\Omega_T}\langle\JJJ, \zzeta\rangle + \eps_0 \int_\Omega \langle\EEE^0, \zzeta(0, \cdot)\rangle,\label{eq:weak_sol2}\\
-\mu_0\int_{\Omega_T}\langle\HHH, \zzeta_t \rangle + \int_{\Omega_T}\langle\EEE, \nabla \times \zzeta \rangle &= - \mu_0\int_{\omega_T} \langle\mmm_t, \zzeta\rangle+ \mu_0 \int_\Omega \langle\HHH^0, \zzeta(0, \cdot)\rangle;\label{eq:weak_sol3} 
\end{align}
\item[(iii)] there holds $\mmm(0,\cdot) = \mmm^0$ in the sense of traces;
\item[(iv)] for almost all $t' \in (0,T)$, we have bounded energy
\begin{align}\label{eq:energy}
\norm{\nabla \mmm(t')}{\L^2(\omega)}^2 + \norm{\mmm_t}{\L^2(\omega_t')}^2 &+ \norm{\HHH(t')}{\L^2(\Omega)}^2 + \norm{\EEE(t')}{\L^2(\Omega)}^2 
 \le C,
\end{align}
where $C>0$ is independent of $t$.
\end{itemize}
\end{definition}
Existence of weak solutions was first shown in~\cite{carbou}. We note, however, that our analysis is constructive in the sense that it also proves existence. 

\begin{remark}
Under additional assumptions on the general contribution $\pi(\cdot)$, namely that $\pi(\cdot)$ is self-adjoint with            
$\norm{\pi(\nnn)}{\L^4(\omega)}\le C$ for all $\nnn \in \L^2(\omega)$ with $|\nnn| \le 1$ almost everywhere, the energy estimate~\eqref{eq:energy} can be improved. The same techniques as in~\cite[Lemma A.1]{multiscale} then show for almost all $t' \in (0,T)$ and $\eps >0$
\begin{align*}
\EE(\mmm, \HHH, \EEE)(t') + 2(\alpha-\eps) \mu_0 \norm{\mmm_t}{\L^2(\omega_{t'})}^2 + 2\sigma \norm{\EEE}{\L^2(\omega_{t'})}^2  \le \EE(\mmm, \HHH, \EEE)(0) - \int_0^{t'} (\JJJ, \EEE),
\end{align*}
where
\begin{align*}
\EE(\mmm, \HHH, \EEE) := \mu_0 C_e \norm{\nabla \mmm}{\L^2(\omega)}^2 + \mu_0 \norm{\HHH}{\L^2(\Omega)}^2 + \eps_0 \norm{\EEE}{\L^2(\omega)}^2 - \mu_0 \langle \pi(\mmm), \mmm \rangle.
\end{align*}
This is in analogy to~\cite{banas}. In particular, the above assumptions are fulfilled in case of vanishing applied field $\HHH_{ext} \equiv 0$ and if $\pi(\cdot)$ denotes the uniaxial anisotropy density.
\end{remark}
\section{Preliminaries}\label{sec:prelim}
For time discretization, we impose a uniform partition $0 = t_0 < t_1 < \hdots < t_N = T$ of the time interval $[0,T]$. The timestep size is denoted by $k=k_j := t_{j+1} - t_j$ for $j = 0, \hdots, N-1$. For each (discrete) function $\vphi, \vphi^j := \vphi(t_j)$ denotes the evaluation at time $t_j$. 
Furthermore, we write $d_t \vphi^{j+1} := (\vphi^{j+1} - \vphi^j)/k$ for $j \ge 1$, and $\vphi^{j+1/2} := (\vphi^{j+1} + \vphi^j)/2$ for $j \ge 0$ and a sequence $\{\vphi^j\}_{j\ge 0}$.

For the spatial discretization, let $\TT_h^{\Omega}$ be a regular triangulation of the polyhedral bounded Lipschitz domain $\Omega \subset \R^3$ into compact and non-degenerate tetrahedra. By $\TT_h$, we denote its restriction to $\omega \subsub \Omega$, where we assume that $\omega$ is resolved, i.e.\
\begin{align*}
\TT_h = \TT_h^{\Omega}|_\omega = \set{T \in \TT_h^{\Omega}}{T \cap \omega \neq \emptyset} \quad \text{and} \quad \overline \omega = \bigcup_{T \in \TT_h} T.
\end{align*}
By $\SS^1(\TT_h)$ we denote the standard $\PP^1$-FEM space of globally continuous and piecewise affine functions from $\omega$ to $\R^3$
\begin{align*}
\SS^1(\TT_h) := \{\pphi_h \in C(\overline \omega, \R^3) : \pphi_h|_K \in \PP_1(K) \text{ for all } K \in \TT_h\}.
\end{align*}
By $\II_h: C(\Omega) \to \SS^1(\TT_h)$, we denote the nodal interpolation operator onto this space.
Now, let the set of nodes of the triangulation $\TT_h$ be denoted by $\NN_h$.
For discretization of the magnetization $\mmm$ in the LLG equation~\eqref{eq:mllg1}, we define the set of admissible discrete magnetizations by
\begin{align*}
\MM_h := \{\pphi_h \in \SS^1(\TT_h) : |\pphi_h(\zzz)| = 1 \text{ for all } \zzz \in \NN_h\}.
\end{align*}
Due to the modulus constraint $|\mmm(t)| = 1$, and therefore $\mmm_t\cdot \mmm = 0$ almost everywhere in $\omega_T$, we discretize the time derivative $\vvv(t_j) := \mmm_t(t_j)$ in the discrete tangent space which is defined by
\begin{align*}
\KK_{\pphi_h} := \{\ppsi_h \in \SS^1(\TT_h|_{\omega}) : \ppsi_h(\zzz) \cdot \pphi_h(\zzz) = 0 \text{ for all } \zzz \in \NN_h\}
\end{align*}
for any $\pphi_h \in \MM_h$. For two vectors $\mathbf{x}, \mathbf{y} \in \R^3, \mathbf{x}\cdot \mathbf{y}$ stands for the usual scalar product in $\R^3$.

To discretize Maxwell's equations~\eqref{eq:mllg2}--\eqref{eq:mllg3}, we use conforming ansatz spaces $\XX_h \subset \HHH^0(\curl;\Omega)$, $\YY_h \subset \L^2(\Omega)$ subordinate to $\TT_h^{\Omega}$ which additionally fullfil $\nabla \times \XX_h \subset \YY_h$. In analogy to~\cite{banas}, we choose first order edge elements
\begin{align*}
\XX_h := \{\vphi_h \in \HHH^0(\curl; \Omega) : \vphi_h|_K \in \PP_1(K) \text{ for all } K \in \TT_h^{\Omega}\}
\end{align*}
and piecewise constants
\begin{align*}
\YY_h :=\{\zzeta_h \in \L^2(\Omega) : \zzeta_h|_K \in \PP_0(K) \text{ for all } K \in \TT_h^{\Omega}\},
\end{align*}

cf.\ \cite[Chapter 8.5]{monk}. Associated with $\XX_h,$ let $\II_{\XX_h}: \H^2(\Omega) \to \XX_h$ denote the corresponding nodal FEM interpolator. Moreover, let 
\begin{align*}
\II_{\YY_h} : \L^2(\Omega) \rightarrow \YY_h
\end{align*}
denote the $\L^2$-orthogonal projection, characterized by
\begin{align*}
\quad (\zzeta - \II_{\YY_h}\zzeta, \mathbf{y}_h) = 0 \quad \text{ for all } \zzeta \in \L^2(\Omega) \text{ and } \mathbf{y}_h \in \YY_h.
\end{align*}
By standard estimates, see e.g. \cite{monk, brennerscott}, one derives the approximation properties 
\begin{align}
&\norm{\vphi - \II_{\XX_h}\vphi}{\L^2(\Omega)} + h\norm{\nabla \times (\vphi - \II_{\XX_h}\vphi)}{\L^2(\Omega)} \le C\, h^2\norm{\nabla ^2 \vphi}{\L^2(\Omega)}\label{eq:interpol1}\\
& \norm{\zzeta - \II_{\YY_h}\zzeta}{\L^2(\Omega)} \le C\,h\norm{\zzeta}{\H^1(\Omega)}\label{eq:interpol2}
\end{align}
for all $\vphi \in \H^2(\Omega)$ and $\zzeta \in \H^1(\Omega)$. 

Finally, given two expressions $A$ and $B$, we write $A \lesssim B$ if there exists a constant $c>0$ which is independent of $h$ and $k$, such that $A \le cB$. In the case $A \lesssim B$ and $B \lesssim A$, we write $A \simeq B$.
\section{Numerical algorithms}\label{sec:algo}

We recall that the  LLG equation~\eqref{eq:mllg1} can equivalently be stated by
\begin{align}\label{eq:llg:algo}
\alpha \mmm_t + \mmm \times \mmm_t = \heff-(\mmm \cdot \heff)\mmm
\end{align}
under the constraint $|\mmm| = 1$ almost everywhere in $\Omega_T$.
This formulation will now be used to construct the upcoming numerical schemes. 
Following the approaches of \textsc{Alouges et al.}~\cite{alouges2008, alouges2011} and \textsc{Bruckner et al.} from~\cite{multiscale}, we propose two algorithms for the numerical integration of MLLG, where the first one follows the lines of~\cite{banas}. 

\subsection{MLLG integrators}\label{sec:integrators}
For ease of presentation, we assume that the applied field $\JJJ$ is continuous in time, i.e.\ $\JJJ \in C\big([0,T]; \L^2(\Omega)\big)$ so that $\JJJ^j: = \JJJ(t_j)$ is meaningful. We emphasize, however, that this is not necessary for our convergence analysis.
\begin{algorithm}\label{alg:1}
\begin{itemize}
Input: Initital data $\mmm^0$, $\EEE^0$, and $\HHH^0$, parameter $\theta \in [0,1]$, counter $j = 0$. For all $j = 0, \hdots, N-1$ iterate:
\item[(i)] Compute unique solution $(\vvv_h^j, \EEE_h^{j+1}, \HHH_h^{j+1}) \in (\KK_{\mmm_h^j}, \XX_h, \YY_h)$ such that for all $(\pphi_h, \ppsi_h, \zzeta_h) \in \KK_{\mmm_h^j} \times \XX_h \times \YY_h$ holds
\begin{subequations}\label{eq:mllg_alg1}
\begin{equation}\label{eq:mllg_alg1_1}
\begin{split}
&\alpha(\vvv_h^j, \pphi_h) + \big((\mmm_h^j \times \vvv_h^j), \pphi_h\big) = -C_e \big(\nabla(\mmm_h^j + \theta k \vvv_h^j), \nabla \pphi_h\big) \\&\hspace*{35ex}+ (\HHH_h^{j+1/2}, \pphi_h) + \big(\pi(\mmm_h^j), \pphi_h\big),
\end{split}
\end{equation}
\begin{align}\label{eq:mllg_alg1_2}
&\hspace*{3ex}\eps_0 (d_t \EEE_h^{j+1}, \ppsi_h) - (\HHH_h^{j+1/2}, \nabla \times \ppsi_h) + \sigma(\chi_\omega\EEE_h^{j+1/2}, \ppsi_h) = -(\JJJ^{j+1/2}, \ppsi_h),\\\label{eq:mllg_alg1_3}
&\hspace*{3ex}\mu_0(d_t\HHH_h^{j+1}, \zzeta_h) + (\nabla \times \EEE_h^{j+1/2}, \zzeta_h) = -\mu_0(\vvv_h^j, \zzeta_h).
\end{align}
\end{subequations}
\item[(ii)] Define $\mmm_h^{j+1} \in \MM_h$ nodewise by $\displaystyle \mmm_h^{j+1}(\zzz) = \frac{\mmm_h^j(\zzz) + k\vvv_h^j(\zzz)}{|\mmm_h^j(\zzz) + k\vvv_h^j(\zzz)|}$ for all $\zzz \in \NN_h$
\end{itemize}
\end{algorithm}
For the sake of computational and implementational ease, LLG and Maxwell's equations can be decoupled which leads to only two linear systems per timestep. This modification is explicitely stated in the second algorithm.
\begin{algorithm}\label{alg:2}
\begin{itemize}
Input: Initital data $\mmm^0$, $\EEE^0$, and $\HHH^0$, parameter $\theta \in [0,1]$, counter $j = 0$. For all $j=0, \hdots, N-1$ iterate:
\item[(i)] Compute unique solution $\vvv_h^j \in \KK_{\mmm_h^j}$ such that for all $\pphi_h \in \KK_{\mmm_h^j}$ holds
\begin{subequations}
\begin{equation}\label{eq:mllg_alg2_1}
\begin{split}
&\alpha(\vvv_h^j, \pphi_h) + \big((\mmm_h^j \times \vvv_h^j), \pphi_h\big) = -C_e \big(\nabla(\mmm_h^j + \theta k \vvv_h^j), \nabla \pphi_h\big) \\&\hspace*{35ex}+ (\HHH_h^{j}, \pphi_h) + \big(\pi(\mmm_h^j), \pphi_h\big).
\end{split}
\end{equation}
\item[(ii)] Compute unique solution $(\EEE_h^{j+1}, \HHH_h^{j+1}) \in (\XX_h, \YY_h)$ such that for all $(\ppsi_h, \zzeta_h) \in \XX_h \times \YY_h$ holds
\begin{align}\label{eq:mllg_alg2_2}
&\hspace*{4ex}\eps_0 (d_t \EEE_h^{j+1}, \ppsi_h) - (\HHH_h^{j+1}, \nabla \times \ppsi_h) + \sigma(\chi_\omega\EEE_h^{j+1}, \ppsi_h) = -(\JJJ^{j}, \ppsi_h),\\\label{eq:mllg_alg2_3}
&\hspace*{4ex}\mu_0(d_t\HHH_h^{j+1}, \zzeta_h) + (\nabla \times \EEE_h^{j+1}, \zzeta_h) = -\mu_0(\vvv_h^j, \zzeta_h).
\end{align}
\item[(iii)] Define $\mmm_h^{j+1} \in \MM_h$ nodewise by $\displaystyle \mmm_h^{j+1}(\zzz) = \frac{\mmm_h^j(\zzz) + k\vvv_h^j(\zzz)}{|\mmm_h^j(\zzz) + k\vvv_h^j(\zzz)|}$ for all $\zzz \in \NN_h$.
\end{subequations}
\end{itemize}
\end{algorithm}

\subsection{Unique solvability}\label{sec:solve}
In this brief section, we show that the two above algorithms are indeed well defined and admit unique solutions in each step of the iterative loop. We start with Algorithm~\ref{alg:1}.

\begin{lemma}\label{lem:solve}
Algorithm~\ref{alg:1} is well defined in the sense that  in each step $j = 0,\hdots, N-1$ of the loop, there exist unique solutions $(\mmm_h^{j+1}, \vvv_h^j, \EEE_h^{j+1}, \HHH_h^{j+1})$.
\end{lemma}

\begin{proof}
We multiply the first equation of~\eqref{eq:mllg_alg1} by $\mu_0$ and the second and third equation by some free parameter $\setc{alg1} > 0$
to define the bilinear form
$a^j(\cdot, \cdot)$ on $(\KK_{\mmm_h^j}, \XX_h, \YY_h)$ by
\begin{align*}
a^j\big((\PPHI, \PPSI, \TTHETA),(\pphi, \ppsi, \zzeta)\big) &:= \alpha\mu_0\,(\PPHI,\pphi) + \mu_0\,\big((\mmm_h^j \times \PPHI), \pphi\big) + \mu_0C_e \theta k\,( \nabla \PPHI,\nabla \pphi) - \frac{\mu_0}{2}(\PPHI,\zzeta)\\
&\quad+ \frac{\c{alg1}\eps_0}{k}( \PPSI,\ppsi) -\frac{\c{alg1}}{2}(\PPSI, \nabla \times \zzeta) + \frac{\c{alg1}\sigma}{2}(\chi_w \PPSI, \ppsi) \\
&\quad+ \frac{\c{alg1}\mu_0}{k}(\TTHETA, \zzeta) +\frac{\c{alg1}}{2} (\nabla \times \TTHETA, \ppsi) + \c{alg1}\mu_0\,(\TTHETA, \pphi)
\end{align*}

and the linear functional $L^j(\cdot)$ on $(\KK_{\mmm_h^j}, \XX_h, \YY_h)$ by
\begin{align*}
L^j\big((\pphi, \ppsi, \zzeta)\big) &:=-\mu_0C_e \, (\nabla \mmm_h^j, \nabla \pphi) +  \frac{\mu_0}{2} (\HHH_h^j, \pphi) +\mu_0 \,\big(\pi(\mmm_h^j),\pphi\big) \\
&\quad -\c{alg1}(\JJJ^{j+1/2},\ppsi) - \frac{\c{alg1}\eps_0}{k}(\EEE_h^j, \ppsi) + \frac{\c{alg1}}{2}(\HHH_h^j, \nabla \times \ppsi) - \frac{\c{alg1}\sigma}{2}(\chi_w\EEE_h^j, \ppsi)\\
&\quad - \frac{\c{alg1}\mu_0}{k} (\HHH_h^j, \zzeta) - \frac{\c{alg1}}{2}(\nabla \times \EEE_h^j, \zzeta).
\end{align*}
To ease the readability, the respective first lines of these definitions stem from~\eqref{eq:mllg_alg1_1}, the second from~\eqref{eq:mllg_alg1_2}, and the third from~\eqref{eq:mllg_alg1_3}.
Clearly, \eqref{eq:mllg_alg1} is equivalent to 
\begin{align*}
a^j\big((\vvv_h^j, \EEE_h^{j+1}, \HHH_h^{j+1}), (\pphi_h, \ppsi_h, \zzeta_h)\big) = L\big((\pphi_h, \ppsi_h, \zzeta_h)\big) \quad \text{ for all } (\pphi_h, \ppsi_h, \zzeta_h) \in \KK_{\mmm_h^j} \times \XX_h \times \YY_h.
\end{align*}
Next, we aim to show that the bilinear form $a^j(\cdot, \cdot)$
is positive definite on $\KK_{\mmm_h^j} \times \XX_h \times \YY_h$. Usage of the H\"older inequality reveals that for all $ (\vphi, \ppsi, \zzeta) \in \KK_{\mmm_h^j} \times \XX_h \times \YY_h$ it holds that 
\begin{align*}
a^j\big((\pphi, \ppsi, \zzeta), (\pphi, \ppsi, \zzeta)\big) &= \alpha\mu_0(\pphi, \pphi) + \mu_0\big((\mmm_h^j \times \pphi), \pphi\big) + \mu_0C_e\theta k\,(\nabla \pphi, \nabla \pphi) - \frac{\mu_0}{2}\,(\pphi, \zzeta)\\
&\quad+ \frac{\c{alg1}\eps_0}{k}(\ppsi, \ppsi) - \frac{\c{alg1}}{2}(\zzeta, \nabla \times \ppsi) + \frac{\c{alg1}\sigma}{2}(\chi_w \ppsi, \ppsi) \\
&\quad+ \frac{\c{alg1}\mu_0}{k}(\zzeta, \zzeta) + \frac{\c{alg1}}{2}(\nabla \times \ppsi, \zzeta) + \c{alg1}\mu_0\,(\pphi, \zzeta)\\
&= \alpha\mu_0(\pphi, \pphi) + \mu_0C_e \theta k\,(\nabla \pphi, \nabla \pphi) + \big(\c{alg1} \mu_0 - \frac{\mu_0}{2}\big)(\pphi, \zzeta)\\
&\quad+ \frac{\c{alg1}\eps_0}{k}(\ppsi, \ppsi) + \frac{\c{alg1}\sigma}{2}(\chi_w \ppsi, \ppsi) + \frac{\c{alg1}\mu_0}{k}(\zzeta, \zzeta)\\
&\ge \underbrace{\big(\alpha -2\eps(\c{alg1}-1/2)\big)}_{=:a}\mu_0\norm{\pphi}{L^2(\omega)}^2 + \frac{\c{alg1} \eps_0}{k}\norm{\ppsi}{L^2(\Omega)}^2 + \underbrace{\big(\frac{\c{alg1}}{k} - \frac{\c{alg1}-1/2}{2\eps}\big)}_{=:b}\mu_0\norm{\zzeta}{L^2(\omega)}^2,
\end{align*}
where we have used
\begin{align*}
(\vphi, \zzeta) \ge - 2\eps\norm{\vphi}{L^2(\omega)}^2 - \frac{1}{2\eps}\norm{\zzeta}{L^2(\omega)}^2.
\end{align*}
In order to prove the desired result, we have to show $a,b > 0$ by choice of $\c{alg1}$. We make the following ansatz: Fix $\setc{alg1C} > k > 0$, choose $\eps > 0$ in such a way that
$0 < \alpha \c{alg1C} - 2 \eps^2 -2 \alpha \eps$, and define $\c{alg1} := \c{alg1C}/(2\c{alg1C}-4\eps)$. Then, Condition $a>0$ is equivalent to
\begin{align*}
0 < a = \alpha -2\eps\big(\frac12 \frac{\c{alg1C}}{\c{alg1C}-2\eps}-\frac12\big) = \alpha -\eps\frac{\c{alg1C}-\c{alg1C}+2\eps}{\c{alg1C}-2\eps} \quad \Longleftrightarrow \quad 0 < \alpha \c{alg1C}-2 \alpha \eps - 2\eps^2
\end{align*}
which is true due to the choice of $\eps$. The condition $b > 0$ now gives
\begin{align*}
0 < b = \frac{\c{alg1}}{k} - \frac{\c{alg1}-\frac12}{2\eps} \quad \Longleftrightarrow \quad k(\c{alg1}-\frac12)< 2\eps \c{alg1} \quad \Longleftrightarrow \quad k < \frac{2\eps \c{alg1}}{\c{alg1}-\frac12} = \c{alg1C},
\end{align*}
which
is automatically fulfilled due to the original choice of $k < \c{alg1C}$. In particular,~\eqref{eq:mllg_alg1} thus admits a unique solution $(\vvv_h^j, \EEE_h^{j+1}, \HHH_h^{j+1})$ in each step of the loop. From  $\vvv_h^j \in \KK_{\mmm_h^j}$ and the Pythagoras theorem, we get $|\mmm_h^j(\zzz)+k \vvv_h^j(\zzz)|^2 = |\mmm_h^j(\zzz)|^2 + k|\vvv_h^j(\zzz)|^2 \ge 1$. Hence, also step $(ii)$ of Algorithm~\ref{alg:1} is well defined. This concludes the proof.
\end{proof}

The following lemma states an analogous result for the second algorithm. 
\begin{lemma}\label{lem:solve:alg2}
Algorithm~\ref{alg:2} is well defined in the sense that it admits a unique solution at each step $j = 0, \hdots, N-1$ of the iterative loop.
\end{lemma}

\begin{proof}
For the first equation~\eqref{eq:mllg_alg2_1}, we define the bilinearform $a^j(\cdot, \cdot) : \KK_{\mmm_h^j} \times \KK_{\mmm_h^j} \rightarrow \R$ by
\begin{align*}
a^j(\PPHI,\pphi) := \alpha\,(\PPHI, \pphi)+ \big((\mmm_h^j \times \PPHI),\pphi \big)
 + \theta C_e k\,(\nabla \PPHI, \nabla \pphi) \end{align*}
and the functional
\begin{align*}
L^j(\pphi) := C_e (\nabla \mmm_h^j, \nabla \pphi) + (\HHH_h^j, \pphi) + (\pi(\mmm_h^j), \pphi).
\end{align*}
Obviously, $L^j(\cdot)$ is linear, while $a^j(\cdot, \cdot)$ is bilinear and positive definite, since
\begin{align*}
a^j(\pphi, \pphi) = \alpha\norm{\pphi}{\L^2(\omega)}^2 + \theta C_e k \norm{\nabla \pphi}{\L^2(\omega)}^2.
\end{align*}
Hence there exists a unique $\vvv_h^j \in \KK_{\mmm_h^j}$ solving \eqref{eq:mllg_alg2_1}. For the second equation \eqref{eq:mllg_alg2_2}, we have to consider the bilinear form $b(\cdot, \cdot) : (\XX_h, \YY_h) \times (\XX_h, \YY_h) \rightarrow \R$ defined by
\begin{align*}
b\big((\PPSI, \TTHETA),(\ppsi, \zzeta)\big) &:= \frac{\eps_0}{k}(\PPSI, \ppsi) - (\PPSI, \nabla \times \zzeta) + \sigma(\chi_w \PPSI, \ppsi) + \frac{\mu_0}{k}(\TTHETA, \zzeta) + (\nabla \times \TTHETA, \ppsi)
\end{align*}
which is continuous and positive definite, since
\begin{align*}
b\big((\pphi, \zzeta)(\pphi, \zzeta)\big) &= \frac{\eps_0}{k}(\pphi, \pphi) - (\zzeta, \nabla \times \pphi) + \sigma(\chi_w \pphi, \pphi) + \frac{\mu_0}{k}(\zzeta, \zzeta) + (\nabla \times \pphi, \zzeta)\\
 &=\frac{\eps_0}{k}(\pphi, \pphi) + \sigma(\chi_w \pphi, \pphi) + \frac{\mu_0}{k}(\zzeta, \zzeta)\\
 & = \frac{\eps_0}{k}\norm{\pphi}{L^2(\Omega)}^2 + \frac{\mu_0}{k}\norm{\zzeta}{L^2(\Omega)}^2+ \sigma \norm{\pphi}{\L^2(\omega)}^2
\end{align*}
and the functional
\begin{align*}
\widetilde L^j\big((\ppsi, \zzeta)\big):= -(\JJJ^j, \ppsi) - \mu_0(\vvv_h^j, \zzeta)
\end{align*}
which is obviously linear.
Due to finite dimension, there is a unique solution $(\EEE_h^{j+1}, \HHH_h^{j+1})$ of~\eqref{eq:mllg_alg2_2}. As in Lemma~\ref{lem:solve}, we see that step $(iii)$ of Algorithm~\ref{alg:2} is also well-defined.
\end{proof}
\section{Main theorem \& Convergence analysis}\label{sec:convergence}
In this section, we aim to show that the two preceeding algorithms indeed define convergent schemes. 
We first consider Algorithm~\ref{alg:2}.

\subsection{Main result}\label{sec:mainresult}
We start by collecting some general assumptions. Throughout, we assume that the spatial meshes $\TT_h|_\omega$ are uniformly shape regular and satisfy the angle condition
\begin{align}\label{eq:assum1}
\int_\omega \nabla \zeta_i \cdot \nabla \zeta_j \le 0 \quad \text{ for all hat functions } \zeta_i, \zeta_j \in \SS^1(\TT_h|_\omega) \text{ with } i \neq j.
\end{align}
For $\xxx \in \Omega$ and $t \in [t_j, t_{j+1})$, we now define for $\gamma_h^\ell \in \{\mmm_h^\ell, \HHH_h^\ell, \EEE_h^\ell, \JJJ^\ell, \vvv_h^\ell\}$ the time approximations
\begin{equation}\label{eq:discrete_functions}
\begin{split}
&\wgamma(t, \xxx) := \frac{t-t_j}{k}\gamma_h^{j+1}(\xxx) + \frac{t_{j+1} - t}{k}\gamma_h^j(\xxx)\\
&\wgamma^-(t, \xxx):= \gamma_h^{j}(\xxx),\quad \wgamma^+(t,\xxx):= \gamma_h^{j+1}(\xxx), \quad \overline \gamma_{hk}(t, \xxx):= \gamma_h^{j+1/2}(\xxx) = \frac{\gamma_h^{j+1}(\xxx)+\gamma_h^j(\xxx)}{2}.
\end{split}
\end{equation}
We suppose that the general energy contribution $\pi(\cdot)$ is uniformly bounded in $\L^2(\omega_T)$, i.e.\
\begin{align}\label{eq:assum2}
\norm{\pi(\nnn)}{\L^2(\omega_T)}^2 \le C_\pi,
\end{align}
with an $(h,k)$-independent constant $C_\pi >0$ for all $\nnn \in \L^2(\omega_T)$ with $\norm{\nnn}{\L^2(\omega_T)}^2 \le 1$ as well as
\begin{align}\label{eq:assum3}
\pi(\nnn_{hk}) \rightharpoonup \pi(\nnn) \quad \text { weakly subconvergent in $\L^2(\omega_T)$}
\end{align}
provided that the sequence $\nnn_{hk} \rightharpoonup \nnn$ is weakly subconvergent in $\H^1(\omega_T)$ towards some $\nnn \in \H^1(\omega_T)$. For the initial data, we assume
\begin{align}\label{eq:assum4}
\mmm_h^0 \rightharpoonup \mmm^0 \quad \text{ weakly in } \L^2(\omega),
\end{align}
as well as
\begin{align}\label{eq:assum5}
\HHH_h^0 \rightharpoonup \HHH^0 \quad \text{ and } \quad  \EEE_h^0 \rightharpoonup \EEE^0  \quad \text{ weakly in } \L^2(\Omega)
\end{align}
Finally, for the field $\JJJ$, we assume sufficient regularity, e.g.\ $\JJJ\in C\big([0,T]; \L^2(\Omega)\big)$, such that 
\begin{align}\label{eq:assum6}
\JJJ^\pm \rightharpoonup \JJJ \quad \text{ weakly in } \L^2(\Omega_T).
\end{align}

\begin{remark}
Before proceeding to the actual proof, we would like to remark on the before mentioned assumptions.
\begin{itemize}
\item[(i)] We emphasize that all energy contributions mentioned in the introduction fulfill the assumptions~\eqref{eq:assum2}--\eqref{eq:assum3} on $\pi(\cdot)$, cf.~\cite{multiscale}.
\item[(ii)]

As in~\cite{multiscale}, the analysis can be extended to include approximations $\pi_h$ of the general field contribution $\pi$. In this case, one needs to ensure uniform boundedness of those approximations as well as the subconvergence property $\pi_h(\nnn_{hk}) \rightharpoonup \pi(\nnn)$ weakly in $\L^2(\omega_T)$ provided $\nnn_{hk}$ is weakly subconvergent to $\nnn$ in $\H^1(\omega_T)$.
\item[(iii)] The angle condition~\eqref{eq:assum1} is a somewhat technical but crucial ingredient for the convergence analysis. Starting from the energy decay relation
\begin{align*}
\int_\omega \big|\nabla \Big(\frac{\mmm}{|\mmm|}\Big)\big|^2 \le \int_\omega |\nabla \mmm|^2,
\end{align*}
it has first been shown in~\cite{bartels},that~\eqref{eq:assum1} and nodewise projection ensures energy decay even on a discrete level, i.e.
\begin{align*}
\int_\omega \big|\nabla \II_h \Big(\frac{\mmm}{|\mmm|}\Big)\big|^2 \le \int_\omega |\nabla \II_h \mmm|^2.
\end{align*} 
This yields the inequality $\norm{\nabla \mmm_h^{j+1}}{\L^2(\omega)}^2 \le \norm{\nabla \mmm_h^j + k\vvv_h^j}{\L^2(\omega)}^2$, which is needed in the upcoming proof.
\item[(iv)]
Note that assumption~\eqref{eq:assum1} is automatically fulfilled for tetrahedral meshes with dihedral angles that are smaller than $\pi/2$. If the condition is satisfied by $\TT_0$, it can be ensured for the refined meshes as well, provided e.g.\ the strategy from~\cite[Section 4.1]{verfurth} is used for refinement.
\end{itemize}
\end{remark}

The next statement is the main theorem of this work.

\begin{theorem}[Convergence theorem]\label{thm:convergence}
Let $(\mmm_{hk}, \vvv_{hk}, \HHH_{hk}, \EEE_{hk})$ be the quantities obtained by either Algorithm~\ref{alg:1} or~\ref{alg:2} and assume~\eqref{eq:assum1}--\eqref{eq:assum6} and $\theta \in (1/2, 1]$.
Then, as $(h,k) \rightarrow (0,0)$ independently of each other, a subsequence of $(\mmm_{hk}, \HHH_{hk}, \EEE_{hk})$ converges weakly in $\H^1(\omega_T) \times \L^2(\Omega_T)\times \L^2(\Omega_T)$ to a weak solution $(\mmm, \HHH, \EEE)$ of MLLG. In particular, each accumulation point of $(\mmm_{hk}, \HHH_{hk}, \EEE_{hk})$ is a weak solution of MLLG in the sense of Definition~\ref{def:weak_sol}.
\end{theorem}

The proof will roughly be done in three steps for either algorithm:
\begin{enumerate}
\item[(i)] Boundedness of the discrete quantities and energies.
\item[(ii)] Existence of weakly convergent subsequences.
\item[(iii)] Identification of the limits as weak solutions of MLLG.
\end{enumerate}

Throughout the proof, we will apply the following discrete version of Gronwall's inequality.
\begin{lemma}[Gronwall]\label{lem:dis_gronwall}
Let $k_0, \hdots, k_{r-1}>0$ and $a_0, \hdots, a_{r-1}, b, C>0$, and let those quantities fulfill 
$
a_0 \le b \text{ and } a_\ell \le b + C\sum_{j=0}^{\ell-1}k_ja_j$ for $\ell = 1, \hdots, r$.
Then, we have
$
a_\ell \le C \exp\Big(C\sum_{j=0}^{\ell-1}k_j\Big)$ for  $\ell = 1, \hdots, r.
$
\end{lemma}
\subsection{Analysis of Algorithm~\ref{alg:2}}\label{sec:alg2}
As mentioned before, we first show the desired boundedness.

\begin{lemma}\label{lem:stability}
There exists $k_0 > 0$ such that for all $k < k_0$, the discrete quantities $(\mmm_h^j, \EEE_h^{j}, \HHH_h^{j}) \in \MM_h \times \XX_h \times \YY_h$ fulfill 
\begin{equation}\label{eq:discrete_energy}
\begin{split}
\norm{\nabla \mmm_h^j}{\L^2(\omega)}^2 + k\sum_{i=0}^{j-1}\norm{\vvv_h^i}{\L^2(\omega)}^2 &+ \norm{\HHH_h^j}{\L^2(\Omega)}^2 + \norm{\EEE_h^j}{\L^2(\Omega)}^2 + \big(\theta - 1/2\big)k^2 \sum_{i=0}^{j-1}\norm{\nabla \vvv_h^i}{\L^2(\omega)}^2 \\
& + \sum_{i=0}^{j-1}\big(\norm{\HHH_h^{i+1} - \HHH_h^i}{\L^2(\Omega)}^2 + \norm{\EEE_h^{i+1} - \EEE_h^i}{\L^2(\Omega)}^2\big)
 \le \c{en_dis}
\end{split}
\end{equation}
for each $j = 0, \hdots, N$ and some constant $\setc{en_dis}>0$ that only depends on $|\Omega|$, on $|\omega|$, as well as on $C_\pi$.
\end{lemma}
\begin{proof}
For Maxwell's equations, i.e.\ step (iii) of Algorithm~\ref{alg:2}, we choose $(\ppsi_h, \zzeta_h) = (\EEE_h^{i+1}, \HHH_h^{i+1})$ as special pair of test functions and get from~\eqref{eq:mllg_alg2_2}--\eqref{eq:mllg_alg2_3}
\begin{align*}
&\frac{\eps_0}{k} (\EEE_h^{i+1} - \EEE_h^{i}, \EEE_h^{i+1}) - (\HHH_h^{i+1}, \nabla \times \EEE_h^{i+1}) + \sigma(\chi_\omega\EEE_h^{i+1},\EEE_h^{i+1}) = -(\JJJ^i, \EEE_h^{i+1})\quad \text{and}\\
&\frac{\mu_0}{k}(\HHH_h^{i+1} - \HHH_h^i, \HHH_h^{i+1}) + (\nabla \times \EEE_h^{i+1}, \HHH_h^{i+1}) = -\mu_0(\vvv_h^i, \HHH_h^{i+1}).
\end{align*}
Summing up those two equations (and multiplying by $1/C_e$), we therefore see
\begin{equation}\label{eq:insert1}
\begin{split}
\frac{\eps_0}{kC_e} (\EEE_h^{i+1} - \EEE_h^{i}, \EEE_h^{i+1}) &+ \frac{\sigma}{C_e}\norm{\EEE_h^{i+1}}{\L^2(\omega)}^2 + \frac{\mu_0}{kC_e}(\HHH_h^{i+1} - \HHH_h^i, \HHH_h^{i+1}) \\ &\quad=-\frac{\mu_0}{C_e}(\vvv_h^i, \HHH_h^{i}) + \frac{\mu_0}{C_e}(\vvv_h^i, \HHH_h^i - \HHH_h^{i+1}) - \frac{1}{C_e}(\JJJ^i, \EEE_h^{i+1}).
\end{split}
\end{equation}
The LLG equation~\eqref{eq:mllg_alg2_1} is now tested with $\vphi_i = \vvv_h^i \in \KK_{\mmm_h^i}$. We get
\begin{align*}
 \alpha (\vvv_h^i, \vvv_h^i) + \underbrace{\big((\mmm_h^i \times \vvv_h^i), \vvv_h^i\big)}_{=0} = 
 - C_e \big(\nabla(\mmm_h^i + \theta k \vvv_h^i), \nabla \vvv_h^i\big) + (\HHH_h^i, \vvv_h^i) + \big(\pi(\mmm_h^i),\vvv_h^i\big),
\end{align*}
whence
\begin{align*}
\frac{\alpha k}{C_e} \norm{\vvv_h^i}{\L^2(\omega)}^2 + \theta k^2 \norm{\nabla \vvv_h^i}{\L^2(\omega)}^2 = -k(\nabla \mmm_h^i, \nabla \vvv_h^i) + \frac{k}{C_e} (\HHH_h^i, \vvv_h^i) + \frac{k}{C_e}\big(\pi(\mmm_h^i), \vvv_h^i\big).
\end{align*}
Next, we follow the lines of~\cite{alouges2008, alouges2011, multiscale} and use the fact that $\norm{\nabla \mmm_h^{i+1}}{\L^2(\omega)}^2 \le \norm{\nabla(\mmm_h^i + k\vvv_h^i)}{\L^2(\omega)}^2$ stemming from the mesh condition~\eqref{eq:assum1}, cf.~\cite{bartels}, to see
\begin{equation}\label{eq:nabla_m_bounded}
\begin{split}
\frac12\norm{\nabla \mmm_h^{i+1}}{\L^2(\omega)}^2 & \le \frac12 \norm{\nabla \mmm_h^i}{\L^2(\omega)}^2 + k\,(\nabla \mmm_h^i, \nabla \vvv_h^i)+ \frac{k^2}{2}\norm{\nabla \vvv_h^i}{\L^2(\omega)} \\
&= \frac12 \norm{\nabla \mmm_h^i}{\L^2(\omega)}^2 - \big(\theta - 1/2\big)k^2\norm{\nabla \vvv_h^i}{\L^2(\omega)}^2\\
&\quad - \frac{\alpha\, k}{C_e}\norm{\vvv_h^i}{\L^2(\omega)}^2 + \frac{k}{C_e}(\HHH_h^i, \vvv_h^i) + \frac{k}{C_e}\big(\pi(\mmm_h^i), \vvv_h^i\big).
\end{split}
\end{equation}
Multiplying the last estimate by $\mu_0/k$ and adding~\eqref{eq:insert1}, we obtain
\begin{align*}
\frac{\mu_0}{2k}(\norm{\nabla \mmm_h^{i+1}}{\L^2(\omega)}^2 &- \norm{\nabla \mmm_h^i}{\L^2(\omega)}^2)+\big(\theta - 1/2\big)\mu_0 k\norm{\nabla \vvv_h^i}{\L^2(\omega)}^2 + \frac{\alpha \mu_0}{C_e}\norm{\vvv_h^i}{\L^2(\omega)}^2 \\
&+ \frac{\eps_0}{k\, C_e}(\EEE_h^{i+1} - \EEE_h^i, \EEE_h^{i+1})+\frac{\sigma}{C_e}\norm{\EEE_h^{i+1}}{\L^2(\omega)}^2 + \frac{\mu_0}{kC_e}(\HHH_h^{i+1} - \HHH_h^i, \HHH_h^{i+1})\\
&\quad\le \frac{\mu_0}{C_e}(\HHH_h^i - \HHH_h^{i+1}, \vvv_h^i) - \frac{1}{C_e}(\JJJ^i, \EEE_h^{i+1}) + \frac{\mu_0}{C_e}\big(\pi(\mmm_h^i), \vvv_h^i\big).
\end{align*}
Next, we recall Abel's summation by parts, i.e.\ for arbitrary $u_i \in \R$ and $j \ge 0$, there holds
\begin{align}\label{eq:abel_sum}
\sum_{i=1}^j(u_i - u_{i-1},u_i) = \frac12 |u_j|^2 - \frac12|u_0|^2 + \frac12\sum_{i=1}^j|u_i - u_{i-1}|^2.
\end{align}
Multiplying the above equation by $k$, summing up over the time intervals, and exploiting Abel's summation for the $\EEE_h^i$ and $\HHH_h^i$ scalar-products, this yields
\begin{align*}
&\frac{\mu_0}{2} \norm{\nabla \mmm_h^j}{\L^2(\omega)}^2 +\big(\theta - 1/2\big)\mu_0k^2\sum_{i=0}^{j-1}\norm{\nabla \vvv_h^i}{\L^2(\omega)}^2+ \frac{\alpha k\mu_0}{C_e}\sum_{i=0}^{j-1}\norm{\vvv_h^i}{\L^2(\omega)}^2 + \frac{\eps_0}{2C_e}\norm{\EEE_h^j}{\L^2(\Omega)}^2 \\&\quad+ \frac{\eps_0}{2C_e}\sum_{i=0}^{j-1}\norm{\EEE_h^{i+1} - \EEE_h^i}{\L^2(\Omega)}^2
+\frac{k\sigma}{C_e}\sum_{i=0}^{j-1}\norm{\EEE_h^{i+1}}{\L^2(\omega)}^2 + \frac{\mu_0}{2C_e}\norm{\HHH_h^{j}}{\L^2(\Omega)}^2 + \frac{\mu_0}{2C_e}\sum_{i=0}^{j-1}\norm{\HHH_h^{i+1} - \HHH_h^i}{\L^2(\Omega)}^2 \\
&\le \frac{k\mu_0}{C_e}\sum_{i=0}^{j-1}(\HHH_h^i - \HHH_h^{i+1}, \vvv_h^i) - \frac{k}{C_e} \sum_{i=0}^{j-1}(\JJJ^i, \EEE_h^{i+1}) + \frac{\mu_0 k}{C_e}\sum_{i=0}^{j-1}\big(\pi(\mmm_h^i), \vvv_h^i\big)\\
&\quad + \underbrace{\frac{\mu_0}{2} \norm{\nabla \mmm_h^0}{\L^2(\omega)}^2 + \frac{\eps_0}{2C_e}\norm{\EEE_h^0}{\L^2(\Omega)}^2 + \frac{\mu_0}{2C_e}\norm{\HHH_h^0}{\L^2(\Omega)}^2}_{=:\EE_h^0}
\end{align*}
for any $j \in 1, \hdots, N$. 
By use of the inequalities of Young and H\"older, the first part of the right-hand side can be estimated by
\begin{align*}
&\frac{k\mu_0}{C_e}\sum_{i=0}^{j-1}(\HHH_h^i - \HHH_h^{i+1}, \vvv_h^i) - \frac{k}{C_e} \sum_{i=0}^{j-1}(\JJJ^i, \EEE_h^{i+1}) + \frac{\mu_0 k}{C_e}\sum_{i=0}^{j-1}(\pi(\mmm_h^i), \vvv_h^i) \\
&\quad\le\frac{k\mu_0}{C_e}\sum_{i=0}^{j-1}\frac{1}{4\eps}(\norm{\pi(\mmm_h^i)}{\L^2(\omega)}^2 + \norm{\HHH_h^{i+1}- \HHH_h^i}{\L^2(\Omega)}^2) + \frac{\eps \mu_0k}{C_e}\sum_{i=0}^{j-1}\norm{\vvv_h^i}{\L^2(\omega)}^2\\ 
&\qquad + \frac{k}{4 \nu C_e}\sum_{i=0}^{j-1}\norm{\EEE_h^{i+1}}{\L^2(\Omega)}^2 + \frac{\nu k}{C_e}\sum_{i=0}^{j-1}\norm{\JJJ^i}{\L^2(\Omega)}^2,
\end{align*}
for any $\eps, \nu > 0$. The combination of the last two estimates yields

\begin{align*}
&\frac{\mu_0}{2} \norm{\nabla \mmm_h^j}{\L^2(\omega)}^2 +\big(\theta - 1/2\big)\mu_0k^2\sum_{i=0}^{j-1}\norm{\nabla \vvv_h^i}{\L^2(\omega)}^2+ \frac{\alpha k\mu_0}{C_e}\sum_{i=0}^{j-1}\norm{\vvv_h^i}{\L^2(\omega)}^2 + \frac{\eps_0}{2C_e}\norm{\EEE_h^j}{\L^2(\Omega)}^2 \\
&\quad+ \frac{\eps_0}{2C_e}\sum_{i=0}^{j-1}\norm{\EEE_h^{i+1} - \EEE_h^i}{\L^2(\Omega)}^2
+\frac{k\sigma}{C_e}\sum_{i=0}^{j-1}\norm{\EEE_h^{i+1}}{\L^2(\omega)}^2 + \frac{\mu_0}{2C_e}\norm{\HHH_h^{j}}{\L^2(\Omega)}^2 + \frac{\mu_0}{2C_e}\sum_{i=0}^{j-1}\norm{\HHH_h^{i+1} - \HHH_h^i}{\L^2(\Omega)}^2 \\
&\le\frac{\mu_0}{4C_e\eps}k\sum_{i=0}^{j-1}(\norm{\pi(\mmm_h^i)}{\L^2(\omega)}^2 + \norm{\HHH_h^{i+1}- \HHH_h^i}{\L^2(\Omega)}^2) + \frac{\eps \mu_0k}{C_e}\sum_{i=0}^{j-1}\norm{\vvv_h^i}{\L^2(\omega)}^2\\
&\quad + \frac{k}{4 \nu C_e}\sum_{i=0}^{j-1}\norm{\EEE_h^{i+1}}{\L^2(\Omega)}^2 + \frac{\nu k}{C_e}\sum_{i=0}^{j-1}\norm{\JJJ^i}{\L^2(\Omega)}^2 + \EE_h^0.
\end{align*}
Unfortunately, the term $\frac{k}{4 \nu C_e}\sum_{i=0}^{j-1}\norm{\EEE_h^{i+1}}{\L^2(\Omega)}^2$ on the right-hand side cannot be absorbed by the term $\frac{k\sigma}{C_e}\sum_{i=0}^{j-1}\norm{\EEE_h^{i+1}}{\L^2(\omega)}^2$ on the left hand-side, since the latter consists only of contributions on the smaller domain $\omega$. The remedy is to artificially enlarge the first term by 
\begin{align*}
\frac{k}{4 \nu C_e}\sum_{i=0}^{j-1}\norm{\EEE_h^{i+1}}{\L^2(\Omega)}^2 \le \frac{k}{2 \nu C_e}\sum_{i=0}^{j-1}\norm{\EEE_h^{i+1} - \EEE_h^i}{\L^2(\Omega)}^2 + \frac{k}{2 \nu C_e}\sum_{i=0}^{j-1}\norm{\EEE_h^{i}}{\L^2(\Omega)}^2
\end{align*}
and absorb the first sum into the corresponding quantity on the left-hand side. 
With 
\begin{align*}
C_\vvv := \frac{\mu_0 k}{C_e}(\alpha - \eps), \quad C_\HHH := \frac{\mu_0}{2C_e}\big(1 - \frac{k}{2\eps}\big), \quad \text{ and } \quad 
C_\EEE := \frac{1}{2C_e}\big(\eps_0 - \frac{k}{\nu}\big),
\end{align*}
this yields
\begin{align*}
&a_j := \frac{\mu_0}{2} \norm{\nabla \mmm_h^j}{\L^2(\omega)}^2 +\big(\theta - 1/2\big)\mu_0k^2\sum_{i=0}^{j-1}\norm{\nabla \vvv_h^i}{\L^2(\omega)}^2 + C_\vvv\sum_{i=0}^{j-1}\norm{\vvv_h^i}{\L^2(\omega)}^2 + \frac{\eps_0}{2C_e}\norm{\EEE_h^j}{\L^2(\Omega)}^2\\
&\quad+ C_\EEE\sum_{i=0}^{j-1}\norm{\EEE_h^{i+1} - \EEE_h^i}{\L^2(\Omega)}^2
+\frac{k\sigma}{C_e}\sum_{i=0}^{j-1}\norm{\EEE_h^{i+1}}{\L^2(\omega)}^2 + \frac{\mu_0}{2C_e}\norm{\HHH_h^{j}}{\L^2(\Omega)}^2 + C_\HHH\sum_{i=0}^{j-1}\norm{\HHH_h^{i+1} - \HHH_h^i}{\L^2(\Omega)}^2\\
&\le \underbrace{\EE_h^0  + \frac{k\mu_0}{4C_e\eps}\sum_{i=0}^{j-1}\norm{\pi(\mmm_h^i)}{\L^2(\omega)}^2 + \frac{\nu k}{C_e}\sum_{i=0}^{j-1}\norm{\JJJ^i}{\L^2(\Omega)}^2}_{=:b} + \frac{k}{2\nu C_e}\sum_{i=0}^{j-1}\norm{\EEE_h^i}{\L^2(\Omega)}^2,\\
&\le b+ \frac{k}{\nu \eps_0} \sum_{i=0}^{j-1}a_i.
\end{align*}

In order to show the desired result, we have to ensure that there are choices of $\eps$ and $\nu$, such that the constants $C_\vvv, C_\HHH$, and $C_\EEE$ are positive, i.e.\
\begin{align*}
(\alpha - \eps)> 0, \quad \big(1 - \frac{k}{2\eps}\big) > 0, \quad \text{ and } \quad \big(\eps_0 - \frac{k}{\nu}\big) > 0
\end{align*}
which is equivalent to $k_0/2 < \eps < \alpha$ and $\nu > k_0/\eps_0$. The application of the discrete  of Gronwall inequality, from Lemma~\ref{lem:dis_gronwall} yields $a_j \le M$ and
thus proves the desired result.
\end{proof}

We can now conclude the existence of weakly convergent subsequences.


\begin{lemma}\label{lem:subsequences}
There exist functions $(\mmm, \HHH, \EEE) \in \H^1(\omega_T, \SSS^2)\times \L^2(\Omega_T) \times \L^2(\Omega_T)$ such that
\begin{subequations}
\begin{align}
&\mmm_{hk} \rightharpoonup \mmm \text{ in } \H^1(\omega_T),\\
&\mmm_{hk}, \mmm_{hk}^\pm, \overline\mmm_{hk} \rightharpoonup \mmm \text{ in } \L^2(\H^1(\omega)),  \label{eq:subsequences1}\\
&\mmm_{hk}, \mmm_{hk}^\pm, \overline\mmm_{hk} \rightarrow \mmm \text{ in } \L^2(\omega_T),\label{eq:subsequences2}\\
&\HHH_{hk}, \HHH_{hk}^\pm, \overline\HHH_{hk} \rightharpoonup \HHH \text{ in } \L^2(\Omega_T),\label{eq:subsequences3}\\
&\EEE_{hk}, \EEE_{hk}^\pm, \overline\EEE_{hk} \rightharpoonup \HHH \text{ in } \L^2(\Omega_T),\label{eq:subsequences4}
\end{align}
\end{subequations}
where the subsequences are succesively constructed, i.e.\
for arbitrary mesh-sizes $h \rightarrow 0$ and timestep-sizes $k \rightarrow 0$ there exist subindices $h_\ell, k_\ell$ for which the above convergence properties are satisfied simultaniously.
In addition, there exist some $\vvv \in \L^2(\omega_T)$ with 
\begin{align}
\vvv_{hk}^- \rightharpoonup \vvv \text{ in } \L^2(\omega_T)
\end{align}
for the same subsequence as above.
\end{lemma}
\begin{proof}
From Lemma~\ref{lem:stability}, we immediately get uniform boundedness of all of those sequences.
A compactness argument thus allows us to succesively extract weakly convergent subsequences. It only remains to show that the corresponding limits coincide, i.e.\
\begin{align*}
\lim \wgamma = \lim \wgamma^- = \lim \wgamma^+ = \lim \overline\gamma_{hk}, \quad \text{ where } \gamma_{hk} \in \{\mmm_{hk}, \HHH_{hk}, \EEE_{hk}\}.
\end{align*}
In particular, Lemma~\ref{lem:stability} provides the uniform bound
\begin{align*}
\sum_{i=0}^{j-1}\norm{\mmm_h^{i+1} - \mmm_h^i}{\L^2(\omega)}^2\le \c{en_dis}.
\end{align*}
Here, we used the fact that $\norm{\mmm_h^{j+1} - \mmm_h^j}{\L^2(\omega)}^2 \le k^2 \norm{\vvv_h^j}{\L^2(\omega)}^2$, see e.g.~\cite{alouges2008} or~\cite[Lemma 3.3.2]{petra}.
We rewrite $\wgamma \in \{\mmm_{hk}, \EEE_{hk}, \HHH_{hk}\}$ as $\gamma_h^j +\frac{t-t_j}{k} (\gamma_h^{j+1} - \gamma_h^j)$ on $[t_{j-1},t_j]$ and thus get
\begin{align*}
\norm{\wgamma - \wgamma^-}{\L^2(\Omega_T)}^2 &=\sum_{j=0}^{N-1} \int_{t_j}^{t_{j+1}}\norm{\gamma_h^j + \frac{t-t_j}{k}(\gamma_h^{j+1} - \gamma_h^j) - \gamma_h^j}{\L^2(\Omega)}^2\\
&\le k\sum_{j=0}^{N-1}\norm{\gamma_h^{j+1}-\gamma_h^j}{\L^2(\Omega)}^2 \longrightarrow 0
\end{align*}
and analogously
\begin{align*}
\norm{\wgamma - \wgamma^+}{\L^2(\Omega_T)}^2 &=\sum_{j=0}^{N-1}\int_{t_j}^{t_{j+1}}\norm{\gamma_h^j + \frac{t-t_j}{k}(\gamma_h^{j+1} - \gamma_h^j) - \gamma_h^{j+1}}{\L^2(\Omega)}^2\\
&\le \sum_{j=0}^{N-1}\int_{t_j}^{t_{j+1}}2\norm{\gamma_h^{j+1} - \gamma_h^j}{\L^2(\Omega)}^2\\
&\le2k\sum_{j=0}^{N-1}\norm{\gamma_h^{j+1} - \gamma_h^j}{\L^2(\Omega)}^2 \longrightarrow 0,
\end{align*}
i.e.\ we have $\lim \wgamma^\pm = \lim \wgamma \in \L^2(\Omega_T)$ resp.\ $\L^2(\omega_T)$. In particular it holds that $\lim \overline \gamma_{hk} = \lim \wgamma$.
From the uniqueness of weak limits and the continuous inclusions $\H^1(\omega_T) \subseteq L^2(\H^1(\omega)) \subseteq \L^2(\omega_T)$, we then even conclude the convergence properties of $\mmm_{hk}, \mmm_{hk}^\pm$, and $\overline\mmm_{hk}$ in $\L^2(\H^1(\omega))$ as well as $\mmm_{hk} \rightharpoonup \mmm$ in $\H^1(\omega_T)$. 
From 
\begin{align*}
\norm{|\mmm| - 1}{\L^2(\omega_T)} \le \norm{|\mmm| - |\mmm_{hk}^-|}{\L^2(\omega_T)} + \norm{|\mmm_{hk}^-| - 1}{\L^2(\omega_T)}
\end{align*}
and 
\begin{align*}
\norm{|\mmm_{hk}^-(t,\cdot)| - 1}{\L^2(\omega)} \le h \max_{t_j} \norm{\nabla \mmm_h^j}{\L^2(\omega)},
\end{align*}
we finally deduce $|\mmm| = 1$ a.e.\ in $\omega_T$.
\end{proof}


\begin{lemma}\label{lem:v=dtm}
The limit function $\vvv \in \L^2(\omega_T)$ equals the time derivative of $\mmm$, i.e.\ $\vvv = \partial_t \mmm$ almost everywhere in $\omega_T$
\end{lemma}
\begin{proof}
The proof follows the lines of~\cite{alouges2008} and we therefore only sketch it. The elaborated arguments can be found in~\cite[Lemma 3.3.12]{petra}. Using the inequality
\begin{align*}
\norm{\partial_t \mmm_{hk} - \vvv_{hk}^-}{\L^1(\omega_T)} \lesssim \frac12 k \norm{\vvv_{hk}^-}{\L^2(\omega_T)}^2,
\end{align*}
we exploit weak semicontinuity of the norm to see
\begin{align*}
\norm{\partial_t \mmm - \vvv}{\L^1(\omega_T)}\le \liminf \norm{\partial_t \mmm_{hk} - \vvv_{hk}^-}{\L^1(\omega_T)} = 0 \quad \text{ as } (h,k) \longrightarrow (0,0),
\end{align*}
whence $\vvv = \partial_t \mmm$ almost everywhere in $\omega_T$.
\end{proof}

\begin{proof}[Proof of Theorem~\ref{thm:convergence}]
For the LLG part of~\eqref{eq:weak_sol1}, we follow the lines of~\cite{alouges2008}. Let $\vphi \in C^\infty(\omega_T)$ and $(\ppsi, \zzeta) \in C_c^\infty\big([0,T);C^\infty(\overline\Omega)\cap \HHH_0(\curl, \Omega)\big),$ be arbitrary. We now define test functions by $(\pphi_h, \ppsi_h, \zzeta_h)(t, \cdot) := \big(\II_h(\mmm_{hk}^-\times \vphi), \II_{\XX_h}\ppsi, \II_{\YY_h}\zzeta\big)(t,\cdot)$. Recall that the $\L^2$-orthogonal projection $\II_{\YY_h}:\L^2(\Omega)\rightarrow \YY_h$ satisfies $(\uuu - \II_{\YY_h}\uuu, \mathbf{y}_h) = 0$ for all $\mathbf{y}_h \in \YY_h$ and all $\uuu \in \L^2(\Omega)$. With the notation~\eqref{eq:discrete_functions}, Equation~\eqref{eq:mllg_alg2_1} of Algorithm~\ref{alg:2} implies
\begin{align*}
\alpha \int_0^T (\vvv_{hk}^-, \pphi_h) + \int_0^T\big((\mmm_{hk}^- \times \vvv_{hk}^-), \pphi_h\big) &= -C_e\int_0^T\big(\nabla(\mmm_{hk}^- + \theta k \vvv_{hk}^-), \nabla \pphi_h)\big)\\
&\qquad \qquad  + \int_0^T (\HHH_{hk}^-, \pphi_h) + \int_0^T\big(\pi(\mmm_{hk}^-), \pphi_h\big)
\end{align*}
With $\pphi_h (t, \cdot) := \II_h(\mmm_{hk}^- \times \vphi)(t, \cdot)$ and the approximation properties of the nodal interpolation operator, this yields

\begin{align*}
\int_0^T &\big((\alpha \vvv_{hk}^- + \mmm_{hk}^- \times\vvv_{hk}^-),(\mmm_{hk}^- \times \vphi)\big) + k\,\theta\int_0^T \big(\nabla \vvv_{hk}^-, \nabla(\mmm_{hk}^- \times \vphi)\big)\\
&\quad+ C_e\int_0^T\big(\nabla \mmm_{hk}^-, \nabla(\mmm_{hk}^- \times \vphi)\big) -\int_0^T \big(\HHH_{hk}^-, (\mmm_{hk}^-\times \vphi)\big)- \int_0^T \big(\pi(\mmm_{hk}^-), (\mmm_{hk}^-\times \vphi)\big)\\
&=\mathcal{O}(h)
\end{align*}

Passing to the limit and using the strong $\L^2(\omega_T)$-convergence of $(\mmm_{hk}^-\times \vphi)$ towards $(\mmm \times \vphi)$, we get
\begin{align*}
\int_0^T \big((\alpha \vvv_{hk}^- + \mmm_{hk}^- \times\vvv_{hk}^-),(\mmm_{hk}^- \times \vphi)\big) &\longrightarrow \int_0^T \big((\alpha \mmm_t + \mmm \times\mmm_t),(\mmm \times \vphi)\big),\\
k\,\theta\int_0^T \big(\nabla \vvv_{hk}^-, \nabla(\mmm_{hk}^- \times \vphi)\big) &\longrightarrow 0, \quad \text{ and }\\\int_0^T\big(\nabla \mmm_{hk}^-, \nabla(\mmm_{hk}^- \times \vphi)\big) &\longrightarrow \int_0^T\big(\nabla \mmm, \nabla(\mmm \times \vphi)\big),
\end{align*}
cf.\ \cite{alouges2008}. For the second limit, we have used the boundedness of $k\norm{\nabla \vvv_{hk}^-}{\L^2(\omega_T)}^2$ for $\theta \in (1/2,1]$, see Lemma~\ref{lem:stability}. The weak convergence properties of $\HHH_{hk}^-$ and $\pi(\mmm_{hk}^-)$ from~\eqref{eq:assum3} now yield
\begin{align*}
\int_0^T \big(\HHH_{hk}^-, (\mmm_{hk}^-\times \vphi)\big) &\longrightarrow \int_0^T \big(\HHH, (\mmm\times \vphi)\big) \quad \text{ and }\\
\int_0^T \big(\pi(\mmm_{hk}^-), (\mmm_{hk}^-\times \vphi)\big) &\longrightarrow \int_0^T \big(\pi(\mmm), (\mmm\times \vphi)\big).
\end{align*}

So far, we thus have proved
\begin{align*}
 \int_0^T \big((\alpha \mmm_t + \mmm \times\mmm_t),(\mmm \times \vphi)\big) & = -C_e\int_0^T\big(\nabla \mmm, \nabla(\mmm \times \vphi)\big)\\
 &+ \int_0^T \big(\HHH, (\mmm\times \vphi)\big) + \int_0^T \big(\pi(\mmm), (\mmm\times \vphi)\big)
\end{align*}
Finally, we use the technical results
\begin{align*}
(\mmm \times \mmm_t) \cdot (\mmm \times \vphi) &= \mmm_t \cdot \vphi,\\
\mmm_t \cdot (\mmm \times \vphi) &= -(\mmm \times \mmm_t) \cdot \vphi, \text{ and }\\
\nabla \mmm \times \nabla(\mmm \times \vphi) &= \nabla \mmm \cdot (\mmm \times \nabla \vphi)
\end{align*} 
for the left-hand side, resp.\ the first term on the right-hand side to conclude~\eqref{eq:weak_sol1}. The equality $\mmm(0, \cdot) = \mmm^0$ in the trace sence follows from the weak convergence $\mmm_{hk} \rightharpoonup \mmm$ in $\H^1(\omega_T)$ and thus weak convergence of the traces. Using the weak convergence $\mmm_h^0 \rightharpoonup \mmm^0$ in $\L^2(\omega)$, we finally identify the sought limit. 
For the Maxwell part~\eqref{eq:weak_sol2}--\eqref{eq:weak_sol3} of Definition~\ref{def:weak_sol}, 
we proceed as in~\cite{banas}. Given the above definition of the testfunctions,~\eqref{eq:mllg_alg2_2} implies
\begin{align*}
&\eps_0 \int_0^T \big((\EEE_{hk})_t, \ppsi_h\big) - \int_0^T(\HHH_{hk}^+, \nabla \times \ppsi_h) + \sigma\int_0^T (\chi_\omega \EEE_{hk}^+, \ppsi_h) = \int_0^T (\JJJ_{hk}^-, \ppsi_h)\\
& \mu_0 \int_0^T \big((\HHH_{hk})_t, \zzeta_h\big) + \int_0^T(\nabla \times \EEE_{hk}^+, \zzeta_h) = -\mu_0 \int_0^T (\vvv_{hk}^-, \zzeta_h).
\end{align*}
We now consider each of those two terms separately. For the first term of the first equation, we integrate by parts in time and get
\begin{align*}
\int_0^T \big((\EEE_{hk})_t, \ppsi_h\big) = -\int_0^T\big(\EEE_{hk}, (\ppsi_h)_t\big) +\underbrace{\big(\EEE_{hk}(T,\cdot), \ppsi_h(T, \cdot)\big)}_{=0} - \big(\EEE_h^0, \ppsi_h(0,\cdot)\big)
\end{align*}
Passing to the limit on the right-hand side, we see
\begin{align}\label{eq:time_der_tmp}
\int_0^T \big((\EEE_{hk})_t, \ppsi_h\big) \longrightarrow -\int_0^T\big(\EEE, \ppsi_t\big) - \big(\EEE^0, \ppsi(0,\cdot)\big),
\end{align}
where we have used the assumed convergence of the initial data. 
For the first term in the second equation we proceed analogously. 
The convergence of the terms
\begin{align*}
\int_0^T(\HHH_{hk}^+, \nabla \times \ppsi_h) &\longrightarrow \int_0^T(\HHH, \nabla \times \ppsi), \\
\int_0^T (\chi_\omega \EEE_{hk}^+, \ppsi_h) & \longrightarrow \int_0^T (\chi_\omega \EEE, \ppsi), \\
\int_0^T (\JJJ_{hk}^-, \ppsi_h) &\longrightarrow \int_0^T (\JJJ, \ppsi), \quad \text{and} \\
\int_0^T (\vvv_{hk}^-, \zzeta_h) &\longrightarrow \int_0^T (\mmm_t, \zzeta)
\end{align*}
is straightforward. Here, we have used the approximation properties~\eqref{eq:interpol1}--\eqref{eq:interpol2} of the interpolation operators for the last two limits. It remains to analyze the second term in the second equation. Using $\nabla \times \EEE_{hk}^+(t) \in \YY_h$ and the orthogonality properties of $\II_{\YY_h}$, we deduce

\begin{align*}
 \int_0^T&(\nabla \times \EEE_{hk}^+, \zzeta_h) =  \int_0^T(\nabla \times \EEE_{hk}^+, \zzeta)  -   \int_0^T\big(\nabla \times \EEE_{hk}^+,(1-\II_{\YY_h}) \zzeta\big)\\
 &=  \int_0^T(\nabla \times \EEE_{hk}^+, \zzeta) =  \int_0^T(\EEE_{hk}^+, \nabla \times \zzeta) \longrightarrow  \int_0^T(\EEE, \nabla \times \zzeta).
\end{align*}

For the last equality, we have used the boundary condition $\zzeta \times \nnn = 0$ on $\partial \Omega_T$ and integration by parts.
This yields~\eqref{eq:weak_sol2} and~\eqref{eq:weak_sol3}.

It remains to show the energy estimate~\eqref{eq:energy}. From the discrete energy estimate~\eqref{eq:discrete_energy}, we get for any $t' \in [0,T]$ with $t' \in [t_j, t_{j+1})$
\begin{align*}
&\norm{\nabla \mmm_{hk}^+(t')}{\L^2(\omega)}^2 + \norm{\vvv_{hk}^-}{\L^2(\omega_{t'})}^2 + \norm{\HHH_{hk}^+(t')}{\L^2(\Omega)}^2
+ \norm{\EEE_{hk}^+(t')}{\L^2(\Omega)}^2\\
&=\norm{\nabla \mmm_{hk}^+(t')}{\L^2(\omega)}^2 + \int_0^{t'} \norm{\vvv_{hk}^-(s)}{\L^2(\omega)}^2 + \norm{\HHH_{hk}^+(t')}{\L^2(\Omega)}^2 
+ \norm{\EEE_{hk}^+(t')}{\L^2(\Omega)}^2\\
&\le \norm{\nabla \mmm_{hk}^+(t')}{\L^2(\omega)}^2 + \int_0^{t_{j+1}} \norm{\vvv_{hk}^-(s)}{\L^2(\omega)}^2 + \norm{\HHH_{hk}^+(t')}{\L^2(\Omega)}^2  
+ \norm{\EEE_{hk}^+(t')}{\L^2(\Omega)}^2\\ 
&\le \c{en_dis}
\end{align*}
Integration in time thus yields for any measurable set $\mathfrak{I} \subseteq [0,T]$
\begin{align*}
&\int_\mathfrak{I}\norm{\nabla \mmm_{hk}^+(t')}{\L^2(\omega)}^2 + \int_\mathfrak{I}\norm{\vvv_{hk}^-}{\L^2(\omega_{t'})}^2 + \int_\mathfrak{I}\norm{\HHH_{hk}^+(t')}{\L^2(\Omega)}^2
 + \int_\mathfrak{I}\norm{\EEE_{hk}^+(t')}{\L^2(\Omega)}^2
 \le \int_\mathfrak{I}\c{en_dis}
\end{align*}
whence weak lower semi-continuity leads to
\begin{align*}
&\int_\mathfrak{I}\norm{\nabla \mmm}{\L^2(\omega)}^2 + \int_\mathfrak{I}\norm{\mmm_t}{\L^2(\omega_{t'})}^2 + \int_\mathfrak{I}\norm{\HHH}{\L^2(\Omega)}^2 
\int_\mathfrak{I}\norm{\EEE}{\L^2(\Omega)}^2 
\le \int_\mathfrak{I}\c{en_dis}.
\end{align*}
The desired result now follows from standard measure theory, see e.g.~\cite[IV, Thm.\ 4.4]{elstrodt}.
\end{proof}

\subsection{Analysis of Algorithm~\ref{alg:1}}\label{sec:alg1}
This section deals with Algorithm~\ref{alg:1} and the analysis follows the lines of Section~\ref{sec:alg2}. 
As before, we first need boundedness of the involved discrete quantities. 

\begin{lemma}\label{lem:stability:alg1}
The discrete quantities $(\mmm_h^j, \EEE_h^{j}, \HHH_h^{j}) \in \MM_h \times \XX_h \times \YY_h$ fulfill 
\begin{equation}\label{eq:discrete_energy_alg1}
\begin{split}
\norm{\nabla \mmm_h^j}{\L^2(\omega)}^2 + k\sum_{i=0}^{j-1}\norm{\vvv_h^i}{\L^2(\omega)}^2 + \norm{\HHH_h^j}{\L^2(\Omega)}^2 + \norm{\EEE_h^j}{\L^2(\Omega)}^2  + \big(\theta - 1/2\big)k^2 \sum_{i=0}^{j-1}\norm{\nabla \vvv_h^i}{\L^2(\omega)}^2 \le \c{en_dis_alg1}
\end{split}
\end{equation}
for each $j = 0, \hdots, N$ and some constant $\setc{en_dis_alg1}>0$ that depends only on $|\Omega|$, $|\omega|$, and $C_\pi$.
\end{lemma}
Note, that in contrast to Lemma~\ref{lem:stability} from the analysis of Algorithm~\ref{alg:2}, we do not have boundedness of $\sum_{i=0}^{j-1}(\norm{\HHH_h^{i+1}-\HHH_h^i}{\L^2(\Omega)}^2 + \norm{\EEE_h^{i+1}-\EEE_h^i}{\L^2(\Omega)}^2)$ in this case.
\begin{proof}
As before, the proof relies on the choice of the correct test functions. For Maxwell's equations~\eqref{eq:mllg_alg1_2}--\eqref{eq:mllg_alg1_3}, we choose $(\ppsi_h, \zzeta_h) = (\EEE_h^{i+1/2}, \HHH_h^{i+1/2})$ and obtain after summing up
\begin{align*}
d_t\big(\frac{\eps_0}{2}\norm{\EEE_h^{i+1}}{\L^2(\Omega)}^2+ \frac{\mu_0}{2}\norm{\HHH_h^{i+1}}{\L^2(\Omega)}^2 \big) &+ \sigma\norm{\chi_\omega \EEE_h^{i+1/2}}{\L^2(\Omega)}^2 \\
&\quad= -(\JJJ^{i+1/2}, \EEE_h^{i+1/2}) - \mu_0(\vvv_h^i, \HHH_h^{i+1/2}).
\end{align*}
For the LLG equation~\eqref{eq:mllg_alg1_1}, we again test with $\vphi_h = \vvv_h^i \in \KK_{\mmm_h^i}$ and argue as in~\eqref{eq:nabla_m_bounded} to see
\begin{align*}
\frac{\mu_0}{2k}\big(\norm{\nabla \mmm_h^{i+1}}{\L^2(\omega)}^2-\norm{\nabla \mmm_h^i}{\L^2(\omega)}^2\big) + \frac{\alpha \mu_0}{C_e}\norm{\vvv_h^i}{\L^2(\omega)}^2 & + \big(\theta-1/2\big)\mu_0 k\norm{\nabla \vvv_h^i}{\L^2(\omega_T)}^2\\
& \le \frac{\mu_0}{C_e}(\HHH_h^{i+1/2}, \vvv_h^i) +\frac{\mu_0}{C_e} \big(\pi(\mmm_h^i), \vvv_h^i\big).
\end{align*}
The combination of the last two estimates thus yields for any $\eps , \nu > 0$
\begin{align*}
d_t\big(\frac{\mu_0}{2} \norm{\nabla \mmm_h^{i+1}}{\L^2(\omega)}^2 &+ \frac{\eps_0}{2C_e}\norm{\EEE_h^{i+1}}{\L^2(\Omega)}^2+\frac{\mu_0}{2C_e}\norm{\HHH_h^{i+1}}{\L^2(\Omega)}^2\big) + \frac{\sigma}{C_e}\norm{\chi_\omega\EEE_h^{i+1/2}}{\L^2(\Omega)}^2\\
&\quad   + \big(\theta-1/2\big)\mu_0 k\norm{\nabla \vvv_h^i}{\L^2(\omega_T)}^2 + \frac{\mu_0}{C_e}(\alpha - \eps)\norm{\vvv_h^i}{\L^2(\omega)}^2\\
&\le \frac{\nu}{C_e}\norm{\JJJ^{i+1/2}}{\L^2(\Omega)}^2 +  \frac{1}{4\nu C_e}\norm{\EEE_h^{i+1/2}}{\L^2(\Omega)}^2 + \frac{\mu_0}{4\eps C_e}\norm{\pi(\mmm_h^i)}{\L^2(\omega)}^2.
\end{align*}
Multiplying by $k$ and summing over the timesteps, we see
\begin{align*}
\frac{\mu_0}{2} \norm{\nabla \mmm_h^{j}}{\L^2(\omega)}^2 &+ \frac{\eps_0}{2C_e}\norm{\EEE_h^{j}}{\L^2(\Omega)}^2+\frac{\mu_0}{2C_e}\norm{\HHH_h^{j}}{\L^2(\Omega)}^2 + \frac{k\sigma}{C_e} \sum_{i=0}^{j-1}\norm{\chi_\omega\EEE_h^{i+1/2}}{\L^2(\Omega)}^2\\
&\quad + \big(\theta-1/2\big)\mu_0 k^2 \sum_{i=0}^{j-1}\norm{\nabla \vvv_h^i}{\L^2(\omega_T)}^2 + \frac{\mu_0 k}{C_e}(\alpha - \eps)\sum_{i=0}^{j-1}\norm{\vvv_h^i}{\L^2(\omega)}^2\\
&\le \frac{k\nu}{C_e} \sum_{i=0}^{j-1}\norm{\JJJ^{i+1/2}}{\L^2(\Omega)}^2 +  \frac{k}{4\nu C_e} \sum_{i=0}^{j-1}\norm{\EEE_h^{i+1/2}}{\L^2(\Omega)}^2 + \frac{k\mu_0}{4\eps C_e} \sum_{i=0}^{j-1}\norm{\pi(\mmm_h^i)}{\L^2(\omega)}^2\\
&\quad + \frac{\mu_0}{2} \norm{\nabla \mmm_h^{0}}{\L^2(\omega)}^2 + \frac{\eps_0}{2C_e}\norm{\EEE_h^{0}}{\L^2(\Omega)}^2+\frac{\mu_0}{2C_e}\norm{\HHH_h^{0}}{\L^2(\Omega)}^2\\
&\le C + \frac{k}{8\nu C_e}\sum_{i=0}^{j-1}\big(\norm{\EEE_h^{i+1}}{\L^2(\Omega)}^2 + \norm{\EEE_h^i}{\L^2(\Omega)}^2\big).
\end{align*}
Analogously to the last section, the term $\sum_{i=0}^{j-1}\norm{\EEE_h^{i+1/2}}{\L^2(\Omega)}^2$ in the third line cannot be absorbed by the one on the left-hand side directly due to the different domains. We thus again extend the quantity by
\begin{align*}
\sum_{i=0}^{j-1}\big(\norm{\EEE_h^{i+1}}{\L^2(\Omega)}^2 + \norm{\EEE_h^i}{\L^2(\Omega)}^2\big) \le \norm{\EEE_h^j}{\L^2(\Omega)}^2 + 2\sum_{i=0}^{j-1}\norm{\EEE_h^i}{\L^2(\Omega)}^2,
\end{align*} 
and absorb the first part by the left-hand side for appropriate $\nu$. As before, the assertion then follows by an application of the discrete Gronwall's inequality.
\end{proof}

Analogously to Lemma~\ref{lem:subsequences}, we conclude the existence of weakly convergent subsequences that fulfill
\begin{subequations}
\begin{align}\label{eq:subsequences_alg1}
&\mmm_{hk} \rightharpoonup \mmm \text{ in } \H^1(\omega_T),\\
&\mmm_{hk}, \mmm_{hk}^\pm, \overline\mmm_{hk} \rightharpoonup \mmm \text{ in } \L^2(\H^1), \quad \\
&\mmm_{hk}, \mmm_{hk}^\pm, \overline\mmm_{hk} \rightarrow \mmm \text{ in } \L^2(\omega_T),\\
&\HHH_{hk}, \HHH_{hk}^{\pm}, \overline \HHH_{hk} \rightharpoonup \HHH \text{ in } \L^2(\Omega_T),\\
&\EEE_{hk}, \EEE_{hk}^\pm, \overline \EEE_{hk} \rightharpoonup \EEE \text{ in } \L^2(\Omega_T),\\
&\vvv_{hk}^- \rightharpoonup \vvv \text{ in } \L^2(\omega_T).
\end{align}
\end{subequations}

Note also that the above mentioned boundedness of 
\begin{align*}
\sum_{i=0}^{j-1}(\norm{\HHH_h^{i+1}-\HHH_h^i}{\L^2(\Omega)}^2 + \norm{\EEE_h^{i+1}-\EEE_h^i}{\L^2(\Omega)}^2)
\end{align*}
is not necessary to prove that the limits of the in time piecewise constant and piecewise affine approximations coincide. The remedy is a clever use of the midpoint rule.
The proof of Theorem~\ref{thm:convergence} for Algorithm~\ref{alg:1} then completely follows the lines of the one for Algorithm~\ref{alg:2}.

\begin{proof}[Proof of the convergence Theorem~\ref{thm:convergence} for Algorithm~\ref{alg:1}]
Again using $(\pphi_h, \ppsi_h, \zzeta_h)(t, \cdot) := \big(\II_h(\mmm_{hk}^-\times \vphi), \II_{\XX_h}\ppsi, \II_{\YY_h}\zzeta\big)(t,\cdot)$ for any $(\vphi, \ppsi, \zzeta) \in C^\infty(\omega_T) \times C_c^\infty\big([0,T);C^\infty(\overline\Omega)\cap \HHH_0(\curl, \Omega)\big)^2$, Algorithm~\ref{alg:1} implies
\begin{equation}\label{eq:proof_eq}
\begin{split}
&\alpha\int_0^T(\vvv_{hk}^-, \pphi_h) + \int_0^T\big((\mmm_{hk}^- \times \vvv_{hk}^-), \pphi_h\big) = -C_e \int_0^T\big(\nabla(\mmm_{hk}^-+ \theta k \vvv_{hk}^-), \nabla \pphi_h\big) \\
&\hspace*{40ex}+ \int_0^T(\overline\HHH_{hk}, \pphi_h) + \int_0^T\big(\pi(\mmm_{hk}^-), \pphi_h\big)\\
&\eps_0\int_0^T\big((\EEE_{hk})_t, \ppsi_h\big) - \int_0^T(\overline\HHH_{hk}, \nabla \times \ppsi_h) + \sigma\int_0^T(\chi_\omega\overline\EEE_{hk}, \ppsi_h) = -\int_0^T(\overline\JJJ_{hk}, \ppsi_h)\\
&\mu_0\int_0^T\big((\HHH_{hk})_t, \zzeta_h\big) + \int_0^T(\nabla \times \overline\EEE_{hk}, \zzeta_h) = -\mu_0\int_0^T(\vvv_{hk}^-, \zzeta_h).
\end{split}
\end{equation}
Next, we use the fact that the above scalar products containing $\overline\EEE_{hk}$ and $\overline\HHH_{hk}$ can be expressed by means of $\EEE_{hk}$ and $\HHH_{hk}$, respectively, by use of piecewise constant test functions in time. For $\Lambda \in C^\infty(\Omega_T)$, consider the piecewise constant approximation $\Lambda^- \in \PP^0(\II_k, C^\infty(\Omega))$ with $\Lambda^-(t) = \Lambda(t_j)$ for $t \in [t_j, t_{j+1})$. Since the midpoint rule is exact for the (piecewise) affine function $(\HHH_{hk}, \Lambda^-)$, there holds
\begin{align*}
\int_0^T (\overline\HHH_{hk}, \Lambda^-) &= \sum_{j=0}^{N-1}\int_{t_j}^{t_{j+1}}\frac12\big(\HHH_h^{j+1} + \HHH_h^j, \Lambda(t_j)\big) = k \sum_{j=0}^{N-1}(\HHH_{hk}, \Lambda^-)\left(\frac{t_{j+1} + t_j}{2}\right)\\&=\sum_{j=0}^{N-1}\int_{t_j}^{t_{j+1}}\big(\HHH_{hk}, \Lambda^-\big)
= \int_0^T (\HHH_{hk}, \Lambda^-).
\end{align*}
Analogously, we get
\begin{align*}
\int_0^T (\overline\HHH_{hk}, \nabla \times \Lambda^-) & = \int_0^T (\HHH_{hk}, \nabla \times \Lambda^-),\\
\int_0^T (\chi_\omega \overline\EEE_{hk}, \Lambda^-) &= \int_0^T (\chi_\omega \EEE_{hk}, \Lambda^-),\\
\int_0^T (\nabla \times \overline\EEE_{hk}, \Lambda^-) &= \int_0^T (\nabla \times \EEE_{hk}, \Lambda^-).
\end{align*}
Now, let $\pphi_h^-, \ppsi_h^-,\zzeta_h^- \in \PP^0(\II_k)$ denote the in time piecewise constant approximations of $\pphi_h, \ppsi_h,$ and $\zzeta_h$ respectively. We then get
\begin{align*}
\int_0^T (\overline\HHH_{hk}, \pphi_h) &= \int_0^T (\overline\HHH_{hk}, \pphi_h^-) + \int_0^T (\overline\HHH_{hk}, \pphi_h - \pphi_h^-)\\
&= \int_0^T (\HHH_{hk}, \pphi_h^-) + \underbrace{\int_0^T (\overline\HHH_{hk}, \pphi_h - \pphi_h^-)}_{=:a_{hk}}
\end{align*}
In the following, we are going to show $\lim_{(h,k)\to (0,0)} a_{hk} = 0$. By definition, we get
\begin{align*}
a_{hk} &= \int_0^T \big(\overline\HHH_{hk}, \II_h(\mmm_{hk}^- \times \vphi - \mmm_{hk}^-\times \vphi^-)\big)\\
&= \int_0^T \big(\overline\HHH_{hk}, (\mmm_{hk}^- \times \vphi - \mmm_{hk}^-\times \vphi^-)\big) + \int_0^T \big(\overline\HHH_{hk}, (1-\II_h)(\mmm_{hk}^- \times \vphi - \mmm_{hk}^-\times \vphi^-)\big).
\end{align*}
For the second term, we immediately get 
\begin{align*}
 \int_0^T \big(\overline\HHH_{hk}, (1-\II_h)(\mmm_{hk}^- \times \vphi - \mmm_{hk}^-\times \vphi^-)\big) = \mathcal{O}(h^2)
\end{align*}
due to boundedness of $\norm{\overline\HHH_{hk}}{\L^2(\Omega)}, \norm{\mmm_{hk}(\cdot)}{\H^1(\Omega)}$ and $\norm{\vphi}{W^{2,\infty}}$, and therefore elementwise boundedness of $\norm{(\mmm_{hk}^- \times \vphi) - (\mmm_{hk}^-\times \vphi^-)}{\H^2(T)}$ for any $TÊ\in \TT_h$. For the first term on the right-hand side, the mean value theorem yields
\begin{align*}
\int_{t_j}^{t_{j+1}}\norm{(\mmm_{hk}^- \times \vphi) - (\mmm_{hk}^-\times \vphi^-)}{\L^2(\Omega)}^2 &\le \norm{\vphi^--\vphi}{\L^\infty([t_j, t_{j+1}];\L^2(\Omega))}^2 \int_{t_j}^{t_{j+1}}\norm{\mmm_{hk}^-}{\L^2(\Omega)}^2\\
&\le k^2\norm{\vphi_t}{L^\infty([t_j,t_{j+1}]\L^2(\Omega))}^2\int_{t_{j}}^{t_{j+1}}\norm{\mmm_{hk}^-}{\L^2(\Omega)}^2 \longrightarrow 0,
\end{align*}
whence strong $\L^2(\Omega_T)$-convergence of $(\mmm_{hk}^- \times \vphi - \mmm_{hk}^-\times \vphi^-)$ to $0$.
Altogether, we thus get
\begin{align*}
\int_0^T \big(\overline\HHH_{hk}, (\mmm_{hk}^- \times \vphi - \mmm_{hk}^-\times \vphi^-)\big) \le \norm{\overline\HHH_{hk}}{\L^2(\Omega_T)}\norm{(\mmm_{hk}^- \times \vphi) - (\mmm_{hk}^-\times \vphi^-)}{\L^2(\Omega_T)} \longrightarrow 0.
\end{align*}
Analogously, we derive $\pphi_h^- \to \mmm \times \vphi$ strongly in $\L^2(\Omega_T)$ and thus conclude
\begin{align*}
\int_0^T (\overline\HHH_{hk}, \pphi_h) \longrightarrow \int_0^T (\HHH, \mmm \times \vphi) \quad \text{ as } (h,k) \to (0,0).
\end{align*}
Similar arguments show
\begin{align*}
\int_0^T (\overline\HHH_{hk}, \nabla \times \ppsi_h) = \int_0^T (\HHH_{hk}, \nabla \times \ppsi_h^-) + \int_0^T (\overline\HHH_{hk},\nabla \times\big( \ppsi_h - \ppsi_h^-)\big) &\longrightarrow \int_0^T (\HHH, \nabla \times \ppsi),\\
\int_0^T (\chi_\omega \overline\EEE_{hk}, \ppsi_h) = \int_0^T (\chi_\omega\EEE_{hk}, \ppsi_h^-) + \int_0^T (\chi_\omega\overline\EEE_{hk}, \ppsi_h - \ppsi_h^-) &\longrightarrow \int_0^T (\chi_\omega \EEE, \ppsi),\\
\int_0^T (\nabla \times \overline\EEE_{hk}, \zzeta_h) = \int_0^T (\nabla \times \EEE_{hk}, \zzeta_h^-) + \int_0^T (\nabla \times \overline\EEE_{hk}, \zzeta_h - \zzeta_h^-) &\longrightarrow \int_0^T (\EEE,\nabla \times \zzeta). 
\end{align*}
Using the convergence properties~\eqref{eq:subsequences_alg1}, the remainder of the proof follows as for the one of Theorem~\ref{thm:convergence} for Algorithm~\ref{alg:2}. As for the energy estimate, we again utilize weak semi-continuity and the discrete energy estimate~\eqref{eq:discrete_energy_alg1}.
\end{proof}
\section{Numerical examples}\label{sec:numerics}

We study the standard $\mu$-mag benchmark problem no. 4, see \cite{mumag} using Algorithm~2 and Algorithm~3. Here, the effective field consists of the magnetic field $\HHH$ from Maxwell's equation and some constant external field $\HHH_{ext}$, i.e.\ $\pi(\mmm_h^j) = \HHH_{ext}$ for all $j = 1, \hdots, N$.
This problem has been solved previously using the midpoint scheme in \cite{banas_mumag}, and we also use those results for comparison.


Despite the fact that the system (\ref{eq:mllg_alg1})
in Algorithm~2 is linear, for computational reasons it is preferable to solve Maxwell's and LLG equations separately. 
After decoupling, the corresponding linear systems can be solved using dedicated linear solvers.
This leads to a considerable improvement in computational performance, cf.\ \cite{banas}.
In order to decouple the respective equations in (\ref{eq:mllg_alg1}), we employ a simple block Gauss-Seidel algorithm.
For simplicity we set $\sigma \equiv 0$, $\JJJ \equiv \mathbf{0}$.
Assuming the solution $\vvv_h^{j-1}$, $\HHH_h^{j}$, $\EEE_h^j$ is known for a fixed time level $j$,
we set $\GGG_h^{0} = \HHH_h^{j}$, $\FFF_h^0 = \EEE_h^j$, and $\www_h^0 = \vvv_h^{j-1}$ and iterate the following problem over ${\ell}$: \emph{Find $\www_h^\ell, \FFF_h^\ell, \GGG_h^\ell \in \KK_{\mmm_h^j}\times \XX_h \times \YY_h$ such that for all $\pphi_h, \ppsi_h, \zzeta_h \in \KK_{\mmm_h^j}\times \XX_h \times \YY_h$, we have}
\begin{subequations}\label{fixiter}
\begin{equation}
\begin{split}
&\alpha(\www_h^{\ell}, \pphi_h) + \big((\mmm_h^j \times \www_h^j), \pphi_h\big) = -C_e \big(\nabla(\mmm_h^j + \theta k \www_h^{\ell}), \nabla \pphi_h\big) \\&\hspace*{35ex}+ (\GGG_h^{{\ell}-1} + \HHH_{ext}, \pphi_h) 
\end{split}
\end{equation}
\begin{align}
&\hspace*{3ex}\eps_0 \frac{2}{k}( \FFF_h^{{\ell}}, \ppsi_h) - (\GGG_h^{{\ell}}, \nabla \times \ppsi_h)  = \eps_0\frac{2}{k}( \EEE_h^{j}, \ppsi_h)\quad \\
&\hspace*{3ex}\mu_0\frac{2}{k}(\GGG_h^{{\ell}}, \zzeta_h) + (\nabla \times \FFF_h^{{\ell}}, \zzeta_h) = \mu_0\frac{2}{k}(\HHH_h^{j}, \zzeta_h)-\mu_0(\www_h^{\ell}, \zzeta_h)
\end{align}
\end{subequations}
until $\|\www_h^{{\ell}}-\www_h^{{\ell}-1}\|_{\infty} + \|\GGG_h^{{\ell}}-\GGG_h^{{\ell}-1}\|_{\infty} + \norm{\FFF_h^\ell - \FFF_h^{\ell-1}}{\infty}< TOL$. In this setting, $\FFF_h^\ell$ is an approximation of $\EEE_h^{j+1/2}$ and $\GGG_h^\ell$ is an approximation of $\HHH_h^{j+1/2}$, respectively. Therefore, we have 
\begin{align*}
\frac{2}{k}(\FFF_h^\ell - \EEE_h^j) \approx \frac{2}{k}(\EEE_h^{j+1/2}-\EEE_h^j) = \frac{\EEE_h^{j+1}-\EEE_h^j}{k} = d_t\EEE_h^{j+1}.
\end{align*}
Analogous treatment of the $\HHH_h^{j+1/2}$-term thus motivates the above algorithm.
We obtain the solution on the time level $j+1$ as $\vvv_h^{j}=\www_h^{\ell}$,
$\HHH_h^{j+1} = 2\GGG_h^{{\ell}}- \HHH_h^{j}$, $\EEE_h^{j+1} = 2\FFF_h^{{\ell}}- \EEE_h^{j}$.
The linear system (\ref{fixiter}a) is solved using a direct solver, where the constraint on the space $\KK_{\mmm_h^j}$ is realized via a Lagrange multiplier, see~ \cite{mathmod2012}. For the solution of the linear system, (\ref{fixiter}b)--(\ref{fixiter}c) we employ a multigrid preconditioned Uzawa algorithm from \cite{banas}. 

The physical parameters for the computation were
$\mu_0 = 1.25667\times 10^{-6}$, $\varepsilon_0 = 0.88422\times10^{-11}$,
$A = 1.3\times 10^{-11}$, $M_s = 8\times 10^{5}$, $\gamma = 2.211\times 10^5$, $\alpha = 0.02$, 
$\HHH_{ext}= (\mu_0M_s)^{-1}(-24.6, 4.3,0)$, $C_e=2A(\mu_0M_s^2)^{-1}$. Here, $\gamma$ denotes the gyromagnetic ratio, and $M_s$ is the so-called saturation magnetization, see e.g.\ \cite{multiscale}. We set $\theta=1$ in both, Algorithm~\ref{alg:1} and~\ref{alg:2}.
The ferromagnetic domain $\omega = 0.5\times 0.125\times 0.003$ $(\mu \mathrm{m})$
is uniformly partitioned into cubes with dimensions of $(3.90625\times 3.90625\times 3) (n\mathrm{m})$,
each cube consisting of six tetrahedra.
The Maxwell's equations are solved on the domain $\Omega = (4\times 4\times 3.072)$ $(\mu \mathrm{m})$.
The finite element mesh for the domain $\Omega$ is constructed by gradual refinement towards the
ferromagnetic domain $\omega$, see Figure~\ref{fig_mesh}.
We take a uniform timestep
$k = 0.05$ which is two times larger than the time-step required for the midpoint scheme \cite{banas_mumag}.
Note that the scheme admits time-steps up to $k=1$, the smaller time-step has been chosen to attain the desired accuracy.
\begin{figure}[htp!]
\center
\includegraphics[width=0.45\textwidth]{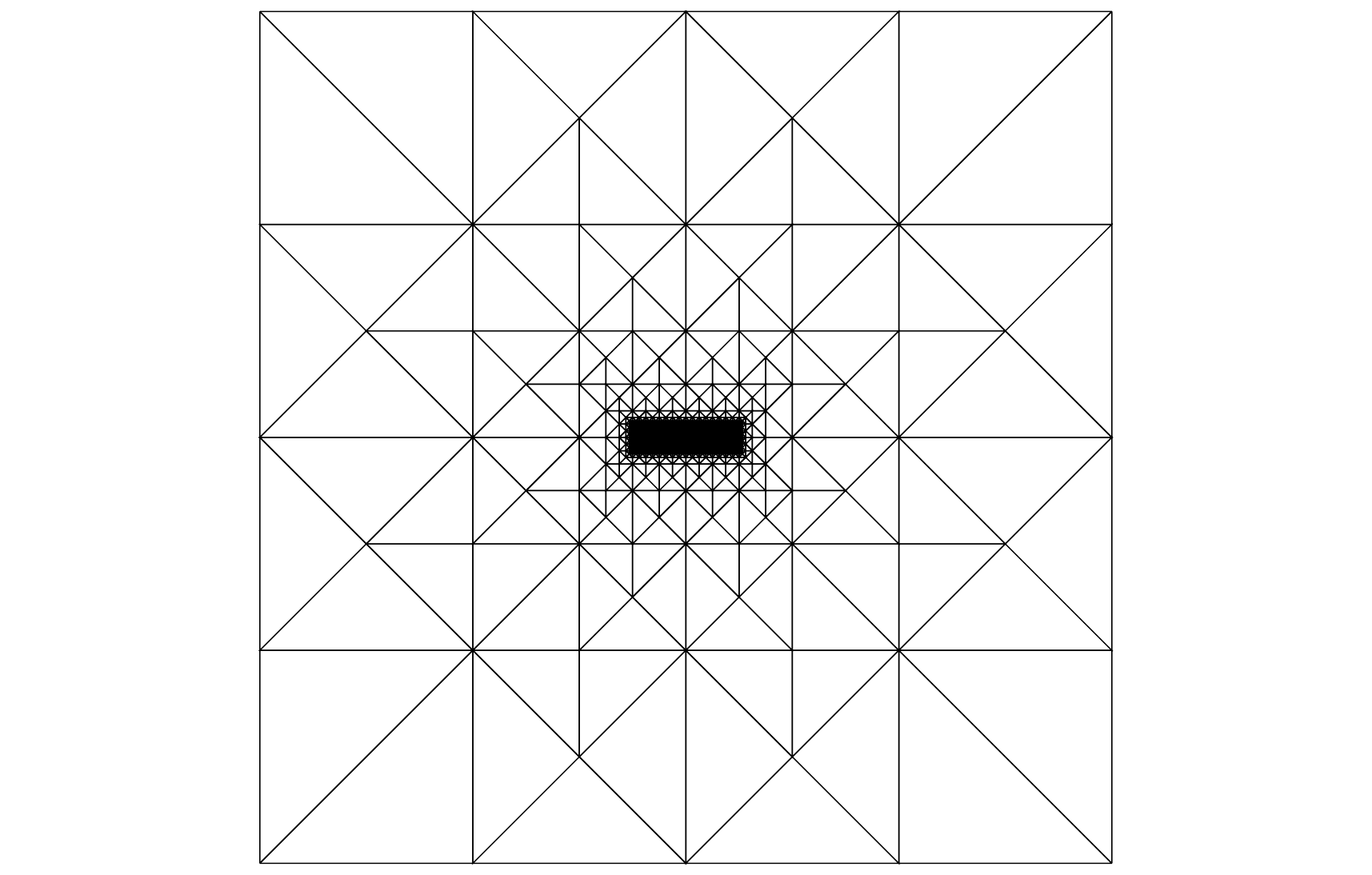}
\includegraphics[width=0.45\textwidth]{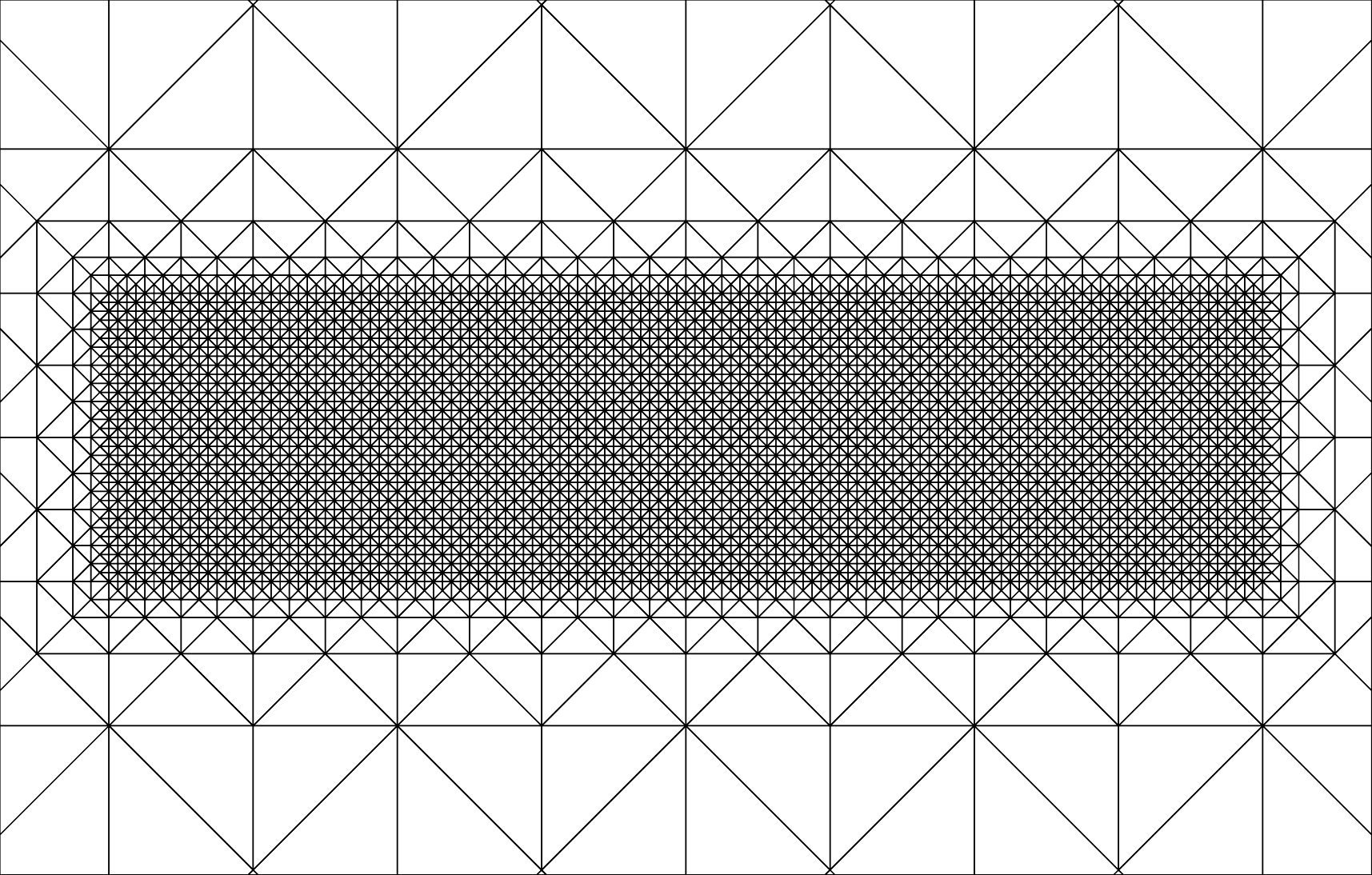}
\caption{Mesh for the domain $\Omega$ at $x_{3} = 0$ (left)
and zoom at the mesh for the domain $\omega$ at $x_3=0$ (right).}
\label{fig_mesh}
\end{figure}

The initial condition $\mmm_0$ for the magnetization is an equilibrium ``S-state'',
see Figure~\ref{fig_sstate}, which is computed from a long-time simulation as in \cite{banas}, \cite{banas_mumag}.
The initial condition $\HHH_0$ is obtained from the magnetostatic approximation
of Maxwell's equation with
and $\EEE_0 = {\bf 0}$, for details see \cite{banas_mumag}.
\begin{figure}[htp!]
\center
\includegraphics[width=0.7\textwidth]{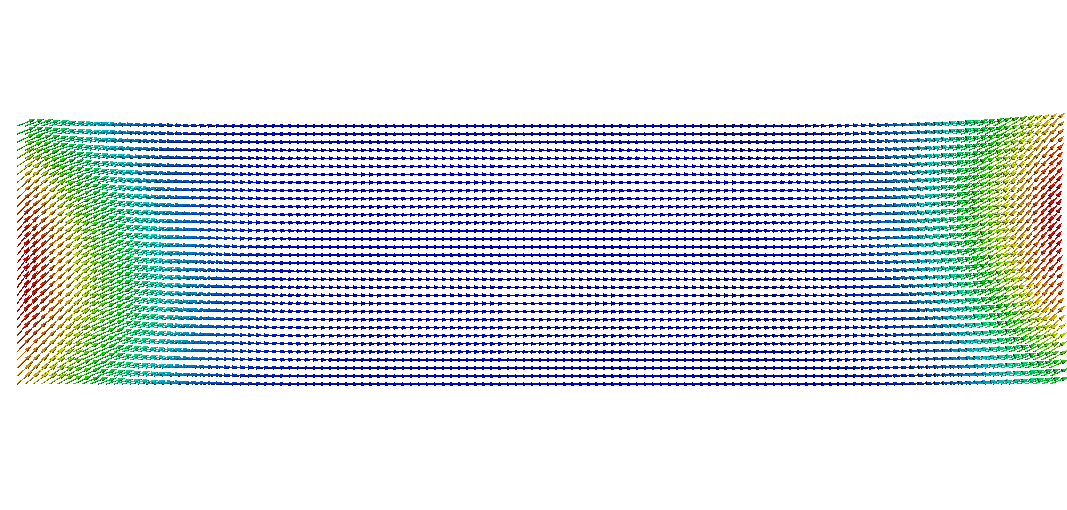}
\caption{Initial condition $\mmm^{0}$.}
\label{fig_sstate}
\end{figure}
In Figure~\ref{fig_averm} we plot the evolution of the average components $m_1$ and $m_2$ of the magnetization
for Algorithm~2 and Algorithm~3.
For comparison, we also present the results  computed
with the midpoint scheme from \cite{banas_mumag} with timestep $k=0.02$.
\begin{figure}[htp!]
\center
\includegraphics[width=0.7\textwidth]{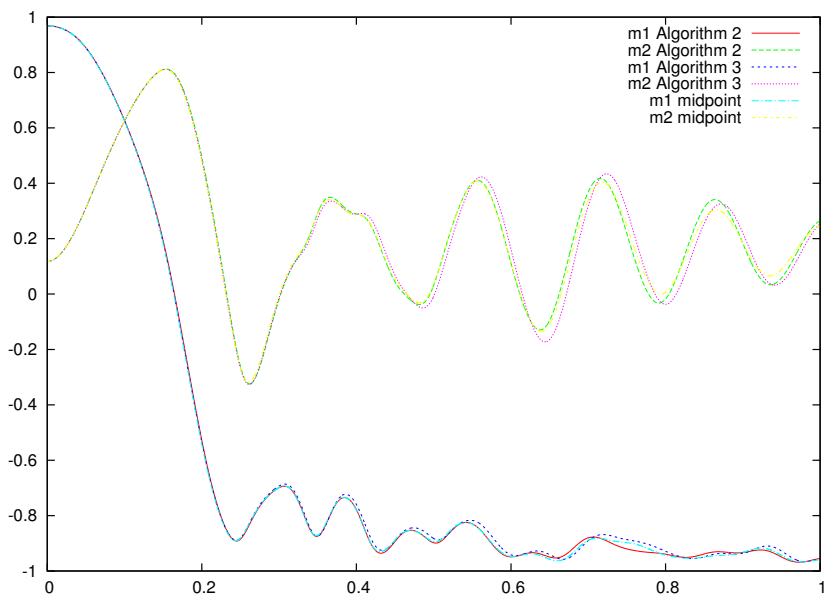}
\caption{Evolution of $|\omega|^{-1}\int_{\omega}m_{1}$ and $|\omega|^{-1}\int_{\omega} m_{2}$, where $m_j$ denotes the $j$-th component of the computed magnetization $\mmm: \omega \to \R^3$.}
\label{fig_averm}
\end{figure}

We also show a snapshot of the magnetization for Algorithm~2 and the midpoint scheme
at times when $|\omega|^{-1}\int_{\omega}m_{1}(t) = 0$ in Figures~\ref{fig_mllg} and \ref{fig_stat},
respectively. We conclude that the results for both algorithms are in good agreement with those computed with the midpoint scheme.
\begin{figure}
\center
\includegraphics[width=0.7\textwidth]{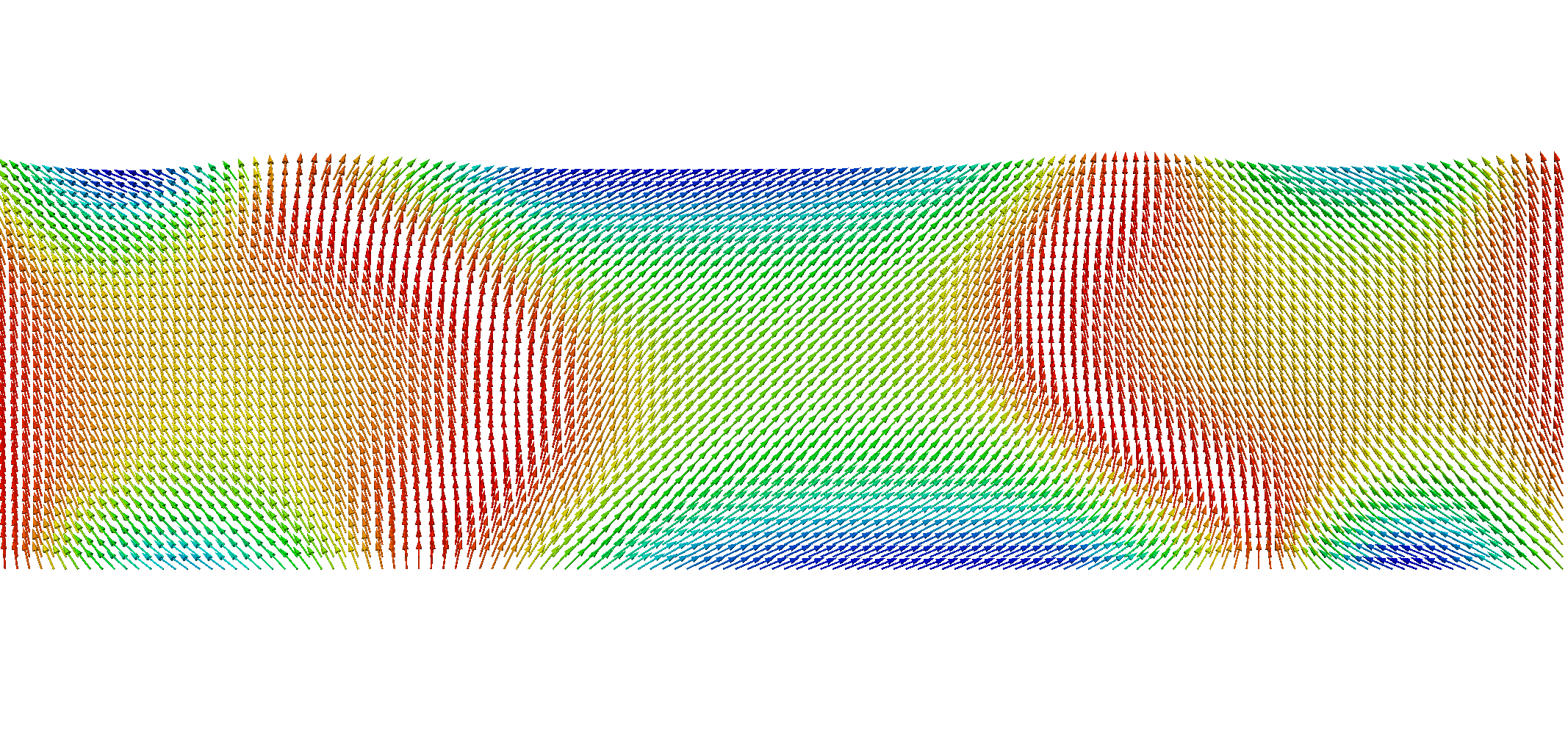}
\caption{Algorithm~\ref{alg:1}: solution at $|\omega|^{-1}\int_{\omega} m_{1}(t) = 0$.}
\label{fig_mllg}
\end{figure}
\vfill
\begin{figure}
\center
\includegraphics[width=0.7\textwidth]{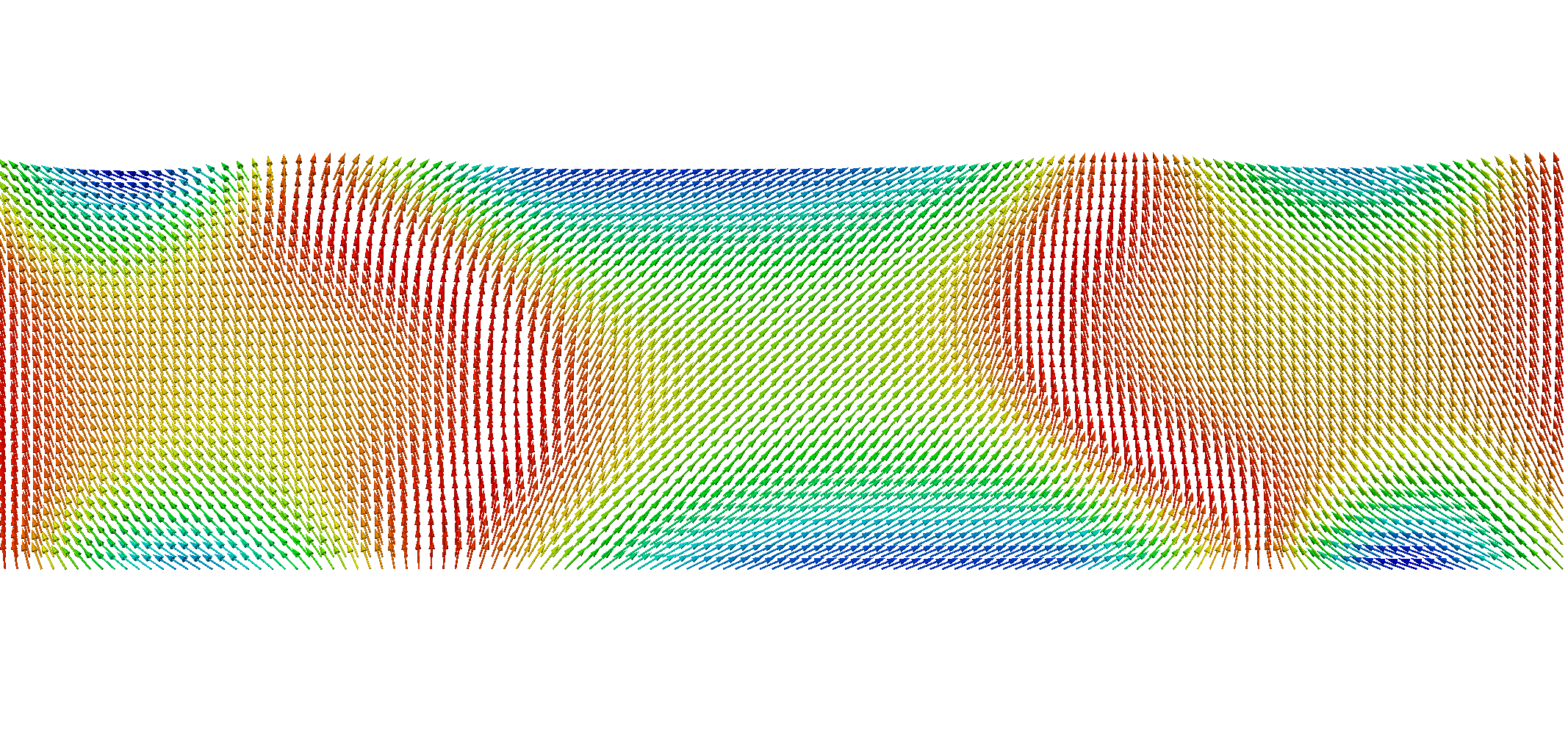}
\caption{Midpoint scheme from~\cite{banas, banas_mumag}: solution at $|\omega|^{-1}\int_{\omega} m_{1} = 0$.}
\label{fig_stat}
\end{figure}


\noindent \textbf{Acknowledgements.}
The authors acknowledge financial support though the WWTF project MA09-029 and the FWF project P21732.

\def\new#1{#1}

\newcommand{\bibentry}[2][!]{\ifthenelse{\equal{#1}{!}}{\bibitem{#2}}{\bibitem[#1]{#2}}}


\end{document}